\documentclass[twoside,11pt,a4paper]{article}
\usepackage[T1]{fontenc}
\usepackage{fancyhdr}
\usepackage{float}
\usepackage{mathrsfs}
\usepackage{amsmath}
\usepackage{amsthm}
\newtheorem{lem}{Lemma}
\usepackage{amssymb}
\usepackage{graphicx}
\usepackage {hyperref}
\numberwithin{equation}{section}
\numberwithin{figure}{section}
\numberwithin{table}{section}
\theoremstyle{remark}

\theoremstyle{plain}
\newtheorem{thm}{\protect\theoremname}[section]
\theoremstyle{plain}
\newtheorem{prop}{\protect\propositionname}[section]

\date{}

\newcommand{\R}{\mathbb{R}}
\newcommand{\p}{\ensuremath{\partial}}
\usepackage{graphics,graphicx}

\providecommand{\notationname}{Notation}
\providecommand{\propositionname}{Proposition}

\providecommand{\theoremname}{Theorem}

\begin{document}
\title{Existence and uniqueness of classical solution to an initial-boundary
value problem for the unsteady general planar Broadwell model with
four velocities}
\author{Koudzo Togbévi Selom SOBAH\textsuperscript{\textsuperscript{1{*}}}\textsuperscript{}
and Amah Séna D'ALMEIDA\textsuperscript{\textsuperscript{2}}}

\maketitle
\textsuperscript{1,2}Department of Mathematics, Faculty of Sciences
and Laboratory of Mathematics and Applications, University of Lomé,
Lomé, TOGO

{*} corresponding email: deselium@gmail.com
\begin{abstract}
We consider the unsteady problem for the general planar Broadwell
model with fourh velocities in a rectangular spatial domain over a
finite time interval. We impose a class of non-negative initial and
Dirichlet boundary data that are bounded and continuous, along with
their first-order partial derivatives. We then prove the existence
and uniqueness of a non-negative continuous solution, bounded together
with its first-order partial derivatives, to the initial-boundary
value problem.
\end{abstract}


\section{Introduction }

The discrete Boltzmann equation (DBE) \cite{key-4} has long served
as a fundamental tool for approximating the behavior of dilute gases,
offering a simplified alternative to the full Boltzmann equation.
Pionner works started with Carleman (1957) \cite{key-9}, Broadwell
(1964) \cite{3,4} and was generalized by Gatignol (1970) \cite{1};
the model replaces the continuous velocity space by a finite set of
discrete velocities, leading to hyperbolic systems of conservation
laws with nonlinear collision terms. These systems have since been
instrumental in the study of various physical and mathematical phenomena,
such as shock formation, entropy dissipation, and boundary layer behavior.

Historically, the initial-boundary value problem (IBVP) for discrete
velocity models has been intensively studied in one-dimensional settings
\cite{5,6,7,11,16,2,9,Cabanne-kawashima}. In the multi-dimensional
steady case, the boundary value problem for discrete velocity models
is investigated, for example, in \cite{10,Nicou=00003D0000E9,d almeida,defoou}.

In contrast, the multi-dimensional theory of the IBVP for DBE models
remains significantly less developed. To our knowledge, no result
had been established for the existence and uniqueness of classical
solutions to these multi-dimensional IBVPs in a bounded domain with
Dirichlet boundary conditions - even in simple geometries such as
rectangles - prior to the recent work \cite{key-1}. In \cite{key-1},
we addressed this gap by providing the proof of existence and uniqueness
of classical solutions for the two-dimensional Broadwell model in
a rectangular domain with Dirichlet boundary conditions.

The goal of the present paper is to generalize these results to the
general four-velocity Broadwell model in the plane, one of the canonical
multi-dimensional discrete models. This work represents a key step
toward a broader mathematical understanding of discrete kinetic models
in higher dimensions.

The paper is organized as follows. In Section \ref{koooapekkke},
the initial-boundary value problem and our main result are presented.
In Section \ref{kooollooaooaiei}, we begin the proof of our result
by the positivity of the solution and in Section \ref{ooolsppikksi},
we end the proof by the existence and uniqueness.

\section{The initial-boundary value problem \label{koooapekkke}}

\subsection{The discrete Boltzmann equations}

Let $d=1,2,3$ and $p$ be a non-negative integer, $\overrightarrow{u_{i}}=\left(u_{i}^{\alpha}\right)_{\alpha=1}^{d}\in\R^{d},$
$\left(i=1,\cdots,p\right),$$A_{ij}^{kl}\geq0$ $\left(i,j,k,l=1,\cdots,p\right)$
.The discrete $p$ velocity model describes a system of particles
with velocities $\overrightarrow{u_{i}}.$ Let $N_{i}\left(t,\left(x_{\alpha}\right)_{\alpha=1}^{d}\right)$
denote the number density of the particles with velocity $\vec{u}_{i},\ i=1,\cdots,p$
at time $t$ and position $M\left(x_{\alpha}\right)_{\alpha=1}^{d}$(
with $\left(x_{\alpha}\right)_{\alpha=1}^{d}\in\R^{d}$). The corresponding
discrete Boltzmann (or kinetic) equations are given by:

\begin{align}
\forall i=1,\cdots,p,\nonumber \\
\dfrac{\partial N_{i}}{\partial t}+\sum_{\alpha=1}^{d}u_{i}^{\alpha}\dfrac{\partial N_{i}}{\partial x_{\alpha}} & =\underbrace{\dfrac{1}{2}\sum_{j,k,l\neq i}A_{ij}^{kl}\left(N_{k}N_{l}\text{\textminus}N_{i}N_{j}\right).}_{\equiv Q_{i}(N)}\label{eq:erttrererer}
\end{align}

In the case of the general four-velocity Broadwell model in the plane,
we have $N_{i}\equiv N_{i}(t,x,y)$ and the kinetic equations are
given by:
\begin{equation}
\begin{cases}
{\textstyle \dfrac{\p N_{1}}{\p t}+c\cos\theta\dfrac{\p N_{1}}{\p x}+c\sin\theta\dfrac{\p N_{1}}{\p y}}=Q\\
\\
\dfrac{\p N_{2}}{\p t}-c\sin\theta\dfrac{\p N_{2}}{\p x}+c\cos\theta\dfrac{\p N_{2}}{\p y}=-Q\\
\\
\dfrac{\p N_{3}}{\p t}+c\sin\theta\dfrac{\p N_{3}}{\p x}-c\cos\theta\dfrac{\p N_{3}}{\p y}=-Q\\
\\
\dfrac{\p N_{4}}{\p t}-c\cos\theta\dfrac{\p N_{4}}{\p x}-c\sin\theta\dfrac{\p N_{4}}{\p y}=Q
\end{cases}
\end{equation}

\[
Q=2cS\left(N_{2}N_{3}-N_{1}N_{4}\right)
\]
$c,S>0$ constants and $0\leq\theta<\dfrac{\pi}{2}.$

\subsection{Initial-boundary value problem }

Let $\left[a_{1},b_{1}\right]\times\left[a_{2},b_{2}\right]$ be a
rectangular spatial domain and $\left[0;T\right]$ a finite time intervall.
We define the space-time domain $\mathscr{P}=\left[0;T\right]\times\left[a_{1};b_{1}\right]\times\left[a_{2};b_{2}\right]$.
In \cite{key-1} we considered the case $\theta=0.$ In the sequel,
$\theta\neq0.$

The initial-boundary value problem we consider is the system $\Sigma$(\ref{eq:erter})-(\ref{eq:mqppa})
:
\begin{align}
{\textstyle \dfrac{\p N_{1}}{\p t}+c\cos\theta\dfrac{\p N_{1}}{\p x}} & {\displaystyle +c\sin\theta\dfrac{\p N_{1}}{\p y}}=Q(N),\;\left(t,x,y\right)\in\mathring{\mathscr{P}}\label{eq:erter}\\
\dfrac{\p N_{2}}{\p t}-c\sin\theta\dfrac{\p N_{2}}{\p x} & +c\cos\theta\dfrac{\p N_{2}}{\p y}=-Q(N),\;\left(t,x,y\right)\in\mathring{\mathscr{P}}\\
\dfrac{\p N_{3}}{\p t}+c\sin\theta\dfrac{\p N_{3}}{\p x} & -c\cos\theta\dfrac{\p N_{3}}{\p y}=-Q(N),\;\left(t,x,y\right)\in\mathring{\mathscr{P}}\\
\dfrac{\p N_{4}}{\p t}-c\cos\theta\dfrac{\p N_{4}}{\p x} & -c\sin\theta\dfrac{\p N_{4}}{\p y}=Q(N),\;\left(t,x,y\right)\in\mathring{\mathscr{P}}\label{eq:zzzfra}\\
N_{i}\left(0,x,y\right)= & N_{i}^{0}\left(x,y\right),\;\left(x,y\right)\in\left[a_{1};b_{1}\right]\times\left[a_{2};b_{2}\right],i=1,\cdots,4\label{eq:lsos}\\
N_{1}\left(t,a_{1},y\right)= & N_{1}^{-}\left(t,y\right),\;\left(t,y\right)\in\left[0;T\right]\times\left[a_{2};b_{2}\right]\label{eq:losso}\\
N_{1}\left(t,x,a_{2}\right)= & N_{1}^{--}\left(t,x\right),\;\left(t,x\right)\in\left[0;T\right]\times\left[a_{1};b_{1}\right]\label{eq:mppos}\\
N_{2}\left(t,b_{1},y\right)= & N_{2}^{+}\left(t,y\right),\;\left(t,y\right)\in\left[0;T\right]\times\left[a_{2};b_{2}\right]\label{eq:mpsss}\\
N_{2}\left(t,x,a_{2}\right)= & N_{2}^{--}\left(t,x\right),\;\left(t,x\right)\in\left[0;T\right]\times\left[a_{1};b_{1}\right]\label{eq:lsoo}\\
N_{3}\left(t,a_{1},y\right)= & N_{3}^{-}\left(t,y\right),\;\left(t,y\right)\in\left[0;T\right]\times\left[a_{2};b_{2}\right]\label{ikis}\\
N_{3}\left(t,x,b_{2}\right)= & N_{3}^{++}\left(t,x\right),\;\left(t,x\right)\in\left[0;T\right]\times\left[a_{1};b_{1}\right]\label{eq:ikioi}\\
N_{4}\left(t,b_{1},y\right)= & N_{4}^{+}\left(t,y\right),\;\left(t,y\right)\in\left[0;T\right]\times\left[a_{2};b_{2}\right]\label{eq:ikioi-1}\\
N_{4}\left(t,x,b_{2}\right)= & N_{4}^{++}\left(t,x\right),\;\left(t,x\right)\in\left[0;T\right]\times\left[a_{1};b_{1}\right]\label{eq:oiiap}
\end{align}
\begin{align}
N_{1}^{0}\left(a_{1},y\right) & =N_{1}^{-}\left(0,y\right),\;y\in\left[a_{2};b_{2}\right]\label{eq:kqiiq}\\
N_{1}^{0}\left(x,a_{2}\right) & =N_{1}^{--}\left(0,x\right),\;x\in\left[a_{1};b_{1}\right]\label{eq:ksos}\\
N_{1}^{-}\left(t,a_{2}\right) & =N_{1}^{--}\left(t,a_{1}\right),\;t\in\left[0;T\right]\label{eq:slos}\\
N_{2}^{0}\left(b_{1},y\right) & =N_{2}^{+}\left(0,y\right),\;y\in\left[a_{2};b_{2}\right]\label{eq:qmpp}\\
N_{2}^{0}\left(x,a_{2}\right) & =N_{2}^{--}\left(0,x\right),\;x\in\left[a_{1};b_{1}\right]\label{eq:lqoop}\\
N_{2}^{+}\left(t,a_{2}\right) & =N_{2}^{--}\left(t,b_{1}\right),\;t\in\left[0;T\right]\label{eq:mqpp}\\
N_{3}^{0}\left(a_{1},y\right) & =N_{3}^{-}\left(0,y\right),\;y\in\left[a_{2};b_{2}\right]\label{eq:lopq}\\
N_{3}^{0}\left(x,b_{2}\right) & =N_{3}^{++}\left(0,x\right),\;x\in\left[a_{1};b_{1}\right]\label{eq:loppq-1}\\
N_{3}^{-}\left(t,b_{2}\right) & =N_{3}^{++}\left(t,a_{1}\right),\;t\in\left[0;T\right]\label{eq:looqp}\\
N_{4}^{0}\left(b_{1},y\right) & =N_{4}^{+}\left(0,y\right),\;y\in\left[a_{2};b_{2}\right]\label{eq:looqp-1}\\
N_{4}^{0}\left(x,b_{2}\right) & =N_{4}^{++}\left(0,x\right),\;x\in\left[a_{1};b_{1}\right]\label{eq:mqpp-1}\\
N_{4}^{+}\left(t,b_{2}\right) & =N_{4}^{++}\left(t,b_{1}\right),\;t\in\left[0;T\right]\label{eq:mqppa}
\end{align}

The initial conditions are given by $N_{i}^{0}$, for $i=1,\ldots,4$,
while the boundary conditions are specified by $N_{1}^{-},N_{1}^{--},N_{2}^{+},N_{2}^{--},N_{3}^{-},N_{3}^{++},N_{4}^{+}$
and $N_{4}^{++}$. All these functions are assumed to be non-negative
and continuous, and their first-order partial derivatives to be continuous
and bounded. (\ref{eq:kqiiq})-(\ref{eq:mqppa}) are compatibility
conditions on the data.

Our problem is to prove the existence and uniqueness of a positive
continuous solution to the system $\Sigma.$ 

We shall use the uniform norm for bounded real-valued functions, and
$\left\Vert f\right\Vert ={\displaystyle \max_{1\leq i\leq4}}\left\Vert f_{i}\right\Vert _{\infty}$
for $f=\left(f_{i}\right)_{i=1}^{4}$ with bounded components; for
bounded real functions of two variables $g\equiv g\left(x,y\right)$
with bounded partial derivatives, we adopt the $C^{1}$ norm $\left\Vert g\right\Vert _{1}\equiv\max\left\{ \left\Vert g\right\Vert _{\infty},\left\Vert \dfrac{\partial g}{\partial x}\right\Vert _{\infty},\left\Vert \dfrac{\partial g}{\partial y}\right\Vert _{\infty}\right\} .$ 

We prove in the sequel the following result 
\begin{thm}
\label{thm:Suppose-.-Then}There exist two positive parameters $p$
and $q$, with $p$ depending on the dimensions of domain $\mathscr{P}$,
and $q$ on the data such that if $pq\leq\dfrac{1}{4}$ then the system
$\Sigma$ admits a unique non-negative continuous solution which is
differentiable except possibly on finitely many planes, with bounded
derivatives, and for which explicit bounds on the solution and its
derivatives are available.
\end{thm}

\section{Non-negativity of the solution\label{kooollooaooaiei}}
\begin{thm}
\label{thm::etaatyeyaiioa-1-1}If they exist the solutions of the
problem $\Sigma$ (eq.(\ref{eq:erter})-(\ref{eq:mqppa})) are non-negative
. 
\end{thm}

\begin{proof}
The proof follows classical arguments ( in \cite{key-1}). Let $\sigma>0.$
Let's put 
\begin{equation}
\rho\left(N\right)=\sum_{i=1}^{4}N_{i}\label{eq:kiid}
\end{equation}
and

\begin{equation}
\begin{cases}
Q_{1}^{\sigma}\left(N\right)=\sigma\rho\left(N\right)N_{1}+Q\left(N\right)\\
Q_{2}^{\sigma}\left(N\right)=\sigma\rho\left(N\right)N_{2}-Q\left(N\right)\\
Q_{3}^{\sigma}\left(N\right)=\sigma\rho\left(N\right)N_{3}-Q\left(N\right)\\
Q_{4}^{\sigma}\left(N\right)=\sigma\rho\left(N\right)N_{4}+Q\left(N\right)
\end{cases}.
\end{equation}
Then for all $\sigma>0,$ $\Sigma$ is equivalent to system $\left(\Sigma_{\sigma}\right)$
\begin{align}
{\textstyle \dfrac{\p N_{1}}{\p t}+c\cos\theta\dfrac{\p N_{1}}{\p x}} & {\displaystyle +c\sin\theta\dfrac{\p N_{1}}{\p y}+\sigma\rho\left(N\right)N_{1}}=Q_{1}^{\sigma}\left(N\right),\;\left(t,x,y\right)\in\mathring{\mathscr{P}}\label{eq:ssdffz-1}\\
\dfrac{\p N_{2}}{\p t}-c\sin\theta\dfrac{\p N_{2}}{\p x} & +c\cos\theta\dfrac{\p N_{2}}{\p y}+\sigma\rho\left(N\right)N_{2}=Q_{2}^{\sigma}\left(N\right),\;\left(t,x,y\right)\in\mathring{\mathscr{P}}\\
\dfrac{\p N_{3}}{\p t}+c\sin\theta\dfrac{\p N_{3}}{\p x} & -c\cos\theta\dfrac{\p N_{3}}{\p y}+\sigma\rho\left(N\right)N_{3}=Q_{3}^{\sigma}\left(N\right),\;\left(t,x,y\right)\in\mathring{\mathscr{P}}\\
\dfrac{\p N_{4}}{\p t}-c\cos\theta\dfrac{\p N_{4}}{\p x} & -c\sin\theta\dfrac{\p N_{4}}{\p y}+\sigma\rho\left(N\right)N_{4}=Q_{4}^{\sigma}\left(N\right),\;\left(t,x,y\right)\in\mathring{\mathscr{P}}\label{eq:qqqqsa-1}
\end{align}

with the conditions (\ref{eq:lsos})-(\ref{eq:mqppa}). 

For $M=\left(M_{1},M_{2},M_{3},M_{4}\right)$ fixed 4-tuple of continuous
functions defined from $\begin{array}{r}
\mathscr{P}\end{array}$ to $\R,$ let's put $\left|M\right|=\left(\left|M_{1}\right|,\left|M_{2}\right|,\left|M_{3}\right|,\left|M_{4}\right|\right)$.
Let's consider the following linear system $\left(\Sigma_{\sigma,M}\right)$
(\ref{eq:ssdffz-1-1})-(\ref{eq:qqqqsa-1-1}) 
\begin{align}
{\textstyle \dfrac{\p N_{1}}{\p t}+c\cos\theta\dfrac{\p N_{1}}{\p x}} & {\displaystyle +c\sin\theta\dfrac{\p N_{1}}{\p y}+\sigma\rho\left(\left|M\right|\right)N_{1}}=Q_{1}^{\sigma}\left(\left|M\right|\right),\;\left(t,x,y\right)\in\mathring{\mathscr{P}}\label{eq:ssdffz-1-1}\\
\dfrac{\p N_{2}}{\p t}-c\sin\theta\dfrac{\p N_{2}}{\p x} & +c\cos\theta\dfrac{\p N_{2}}{\p y}+\sigma\rho\left(\left|M\right|\right)N_{2}=Q_{2}^{\sigma}\left(\left|M\right|\right),\;\left(t,x,y\right)\in\mathring{\mathscr{P}}\label{eq:skiz}\\
\dfrac{\p N_{3}}{\p t}+c\sin\theta\dfrac{\p N_{3}}{\p x} & -c\cos\theta\dfrac{\p N_{3}}{\p y}+\sigma\rho\left(\left|M\right|\right)N_{3}=Q_{3}^{\sigma}\left(\left|M\right|\right),\;\left(t,x,y\right)\in\mathring{\mathscr{P}}\label{eq:losik}\\
\dfrac{\p N_{4}}{\p t}-c\cos\theta\dfrac{\p N_{4}}{\p x} & -c\sin\theta\dfrac{\p N_{4}}{\p y}+\sigma\rho\left(\left|M\right|\right)N_{4}=Q_{4}^{\sigma}\left(\left|M\right|\right),\;\left(t,x,y\right)\in\mathring{\mathscr{P}}\label{eq:qqqqsa-1-1}
\end{align}
with the conditions (\ref{eq:lsos})-(\ref{eq:mqppa}) . It is easy
to see that $\left(\Sigma_{\sigma,M}\right)$ has an unique solution
$N_{M}=\left(N_{1},N_{2},N_{3}\right)$ defined as followed:

for inequalities $f\left(t,x,y\right)\leq0$ and $g\left(t,x,y\right)\leq0$
defined on $\mathscr{P}$, let $\mathbb{I}_{\begin{cases}
f\left(t,x,y\right)\leq0\\
g\left(t,x,y\right)\leq0
\end{cases}}$ denote the identity function of the set 

$\left\{ \left(t,x,y\right)\in\mathscr{P}:f\left(t,x,y\right)\leq0\text{ and }g\left(t,x,y\right)\leq0\right\} ;$
then 
\begin{multline}
N_{1}\left(t,x,y\right)=N_{1}^{A}\left(t,x,y\right)\cdot\mathbb{I}_{\begin{cases}
x-ct\cos\theta\geq a_{1}\\
y-ct\sin\theta\geq a_{2}
\end{cases}}\left(t,x,y\right)+\\
N_{1}^{B}\left(t,x,y\right)\cdot\mathbb{I}_{\begin{cases}
x-ct\cos\theta\leq a_{1}\\
x\sin\theta-y\cos\theta\leq a_{1}\sin\theta-a_{2}\cos\theta
\end{cases}}\left(t,x,y\right)+\\
N_{1}^{C}\left(t,x,y\right)\cdot\mathbb{I}_{\begin{cases}
y-ct\sin\theta\leq a_{2}\\
x\sin\theta-y\cos\theta\geq a_{1}\sin\theta-a_{2}\cos\theta
\end{cases}}\left(t,x,y\right)\label{aoalal}
\end{multline}
where

\begin{multline}
N_{1}^{A}\left(t,x,y\right)=\Biggl({\displaystyle \int_{0}^{t}}e^{\sigma\int_{0}^{s}\rho\left(\left|M\right|\right)\left(r,x+c\left(r-t\right)\cos\theta,y+c\left(r-t\right)\sin\theta\right)dr}\cdot Q_{1}^{\sigma}\left(\left|M\right|\right)\\
\left(s,x+c\left(s-t\right)\cos\theta,y+c\left(s-t\right)\sin\theta\right)ds+N_{1}^{0}\left(x-ct\cos\theta,y-ct\sin\theta\right)\Biggr)\cdot\\
e^{-\sigma{\textstyle {\displaystyle \int_{0}^{t}}}\rho\left(\left|M\right|\right)\left(s,x+c\left(s-t\right)\cos\theta,y+c\left(s-t\right)\sin\theta\right)ds}.\label{eq:olole-1}
\end{multline}

\begin{multline}
N_{1}^{B}\left(t,x,y\right)=\\
\Biggl({\textstyle {\displaystyle \int_{0}^{\frac{1}{c\cos\theta}x-\frac{a_{1}}{c\cos\theta}}}e^{\sigma\int_{0}^{s}\rho\left(\left|M\right|\right)\left(r+t-\frac{1}{c\cos\theta}x+\frac{a_{1}}{c\cos\theta},rc\cos\theta+a_{1},rc\sin\theta-\frac{\sin\theta}{\cos\theta}x+y+\frac{\sin\theta}{\cos\theta}a_{1}\right)dr}}\cdot\\
{\textstyle Q_{1}^{\sigma}\left(\left|M\right|\right)\left(s+t-\frac{1}{c\cos\theta}x+\frac{a_{1}}{c\cos\theta},sc\cos\theta+a_{1},sc\sin\theta-\frac{\sin\theta}{\cos\theta}x+y+\frac{\sin\theta}{\cos\theta}a_{1}\right)ds}+\\
{\textstyle N_{1}^{-}\left(t-\frac{1}{c\cos\theta}x+\frac{a_{1}}{c\cos\theta},-\frac{\sin\theta}{\cos\theta}x+y+\frac{\sin\theta}{\cos\theta}a_{1}\right)\Biggr)}\cdot\\
{\textstyle e^{-\sigma{\displaystyle \int_{0}^{t}}\rho\left(\left|M\right|\right)\left(s+t-\frac{1}{c\cos\theta}x+\frac{a_{1}}{c\cos\theta},sc\cos\theta+a_{1},sc\sin\theta-\frac{\sin\theta}{\cos\theta}x+y+\frac{\sin\theta}{\cos\theta}a_{1}\right)ds}}.\label{eq:olole-1-1}
\end{multline}

\begin{multline}
N_{1}^{C}\left(t,x,y\right)=\\
\Biggl({\textstyle {\displaystyle \int_{0}^{\frac{1}{c\sin\theta}y-\frac{a_{2}}{c\sin\theta}}}e^{\sigma\int_{0}^{s}\rho\left(\left|M\right|\right)\left(r+t-\frac{1}{c\sin\theta}y+\frac{a_{2}}{c\sin\theta},rc\cos\theta+x-\frac{\cos\theta}{\sin\theta}y+\frac{\cos\theta}{\sin\theta}a_{2},rc\sin\theta+a_{2}\right)dr}}\cdot\\
{\textstyle Q_{1}^{\sigma}\left(\left|M\right|\right)\left(s+t-\frac{1}{c\sin\theta}y+\frac{a_{2}}{c\sin\theta},sc\cos\theta+x-\frac{\cos\theta}{\sin\theta}y+\frac{\cos\theta}{\sin\theta}a_{2},sc\sin\theta+a_{2}\right)ds}+\\
{\textstyle N_{1}^{--}\left(t-\frac{1}{c\sin\theta}y+\frac{a_{2}}{c\sin\theta},x-\frac{\cos\theta}{\sin\theta}y+\frac{\cos\theta}{\sin\theta}a_{2}\right)\Biggr)}\cdot\\
{\textstyle e^{-\sigma{\displaystyle \int_{0}^{\frac{1}{c\sin\theta}y-\frac{a_{2}}{c\sin\theta}}}\rho\left(\left|M\right|\right)\left(s+t-\frac{1}{c\sin\theta}y+\frac{a_{2}}{c\sin\theta},sc\cos\theta+x-\frac{\cos\theta}{\sin\theta}y+\frac{\cos\theta}{\sin\theta}a_{2},sc\sin\theta+a_{2}\right)ds}}.\label{eq:olole-1-1-1-1}
\end{multline}
\begin{multline}
N_{2}\left(t,x,y\right)=N_{2}^{A}\left(t,x,y\right)\cdot\mathbb{I}_{\begin{cases}
x-ct\cos\theta\geq a_{1}\\
y-ct\sin\theta\geq a_{2}
\end{cases}}\left(t,x,y\right)+\\
N_{2}^{B}\left(t,x,y\right)\cdot\mathbb{I}_{\begin{cases}
x-ct\cos\theta\leq a_{1}\\
x\sin\theta-y\cos\theta\leq a_{1}\sin\theta-a_{2}\cos\theta
\end{cases}}\left(t,x,y\right)+\\
N_{2}^{C}\left(t,x,y\right)\cdot\mathbb{I}_{\begin{cases}
y-ct\sin\theta\leq a_{2}\\
x\sin\theta-y\cos\theta\geq a_{1}\sin\theta-a_{2}\cos\theta
\end{cases}}\left(t,x,y\right)\label{aoalal-1}
\end{multline}

where
\begin{multline}
N_{2}^{A}\left(t,x,y\right)=\Biggl({\textstyle {\displaystyle \int_{0}^{t}}e^{\sigma\int_{0}^{s}\rho\left(\left|M\right|\right)\left(r,x-c\left(r-t\right)\sin\theta,y+c\left(r-t\right)\cos\theta\right)dr}}\cdot\\
{\textstyle Q_{2}^{\sigma}\left(\left|M\right|\right)\left(s,x-c\left(s-t\right)\sin\theta,y+c\left(s-t\right)\cos\theta\right)ds}+{\textstyle N_{2}^{0}\left(x+ct\sin\theta,y-ct\cos\theta\right)\Biggr)}\cdot\\
{\textstyle e^{-\sigma{\displaystyle \int_{0}^{t}}\rho\left(\left|M\right|\right)\left(s,x-c\left(s-t\right)\sin\theta,y+c\left(s-t\right)\cos\theta\right)ds}}.\label{eq:olole-1-1-1-1-1-1}
\end{multline}

\begin{multline}
N_{2}^{B}\left(t,x,y\right)=\\
\Biggl({\textstyle {\displaystyle \int_{0}^{-\frac{1}{c\sin\theta}x+\frac{b_{1}}{c\sin\theta}}}e^{\sigma\int_{0}^{s}\rho\left(\left|M\right|\right)\left(r+t+\frac{1}{c\sin\theta}x-\frac{b_{1}}{c\sin\theta},-rc\sin\theta+b_{1},rc\cos\theta+\frac{\cos\theta}{\sin\theta}x+y-\frac{\cos\theta}{\sin\theta}b_{1}\right)dr}}\cdot\\
{\textstyle Q_{2}^{\sigma}\left(\left|M\right|\right)\left(s+t+\frac{1}{c\sin\theta}x-\frac{b_{1}}{c\sin\theta},-sc\sin\theta+b_{1},sc\cos\theta+\frac{\cos\theta}{\sin\theta}x+y-\frac{\cos\theta}{\sin\theta}b_{1}\right)ds}+\\
{\textstyle N_{2}^{+}\left(t+\frac{1}{c\sin\theta}x-\frac{b_{1}}{c\sin\theta},\frac{\cos\theta}{\sin\theta}x+y-\frac{\cos\theta}{\sin\theta}b_{1}\right)\Biggr)}\cdot\\
{\textstyle e^{-\sigma{\displaystyle \int_{0}^{-\frac{1}{c\sin\theta}x+\frac{b_{1}}{c\sin\theta}}}\rho\left(\left|M\right|\right)\left(s+t+\frac{1}{c\sin\theta}x-\frac{b_{1}}{c\sin\theta},-sc\sin\theta+b_{1},sc\cos\theta+\frac{\cos\theta}{\sin\theta}x+y-\frac{\cos\theta}{\sin\theta}b_{1}\right)ds}}.\label{eq:olole-1-1-1-1-2}
\end{multline}

\begin{multline}
N_{2}^{C}\left(t,x,y\right)=\\
\Biggl({\textstyle {\displaystyle \int_{0}^{\frac{1}{c\cos\theta}y-\frac{a_{2}}{c\cos\theta}}}e^{\sigma\int_{0}^{s}\rho\left(\left|M\right|\right)\left(r+t-\frac{1}{c\cos\theta}y+\frac{a_{2}}{c\cos\theta},-rc\sin\theta+x+\frac{\sin\theta}{\cos\theta}y-\frac{\sin\theta}{\cos\theta}a_{2},rc\cos\theta+a_{2}\right)dr}}\cdot\\
{\textstyle Q_{2}^{\sigma}\left(\left|M\right|\right)\left(s+t-\frac{1}{c\cos\theta}y+\frac{a_{2}}{c\cos\theta},-sc\sin\theta+x+\frac{\sin\theta}{\cos\theta}y-\frac{\sin\theta}{\cos\theta}a_{2},sc\cos\theta+a_{2}\right)ds}+\\
{\textstyle N_{2}^{--}\left(t-\frac{1}{c\cos\theta}y+\frac{a_{2}}{c\cos\theta},x+\frac{\sin\theta}{\cos\theta}y-\frac{\sin\theta}{\cos\theta}a_{2}\right)\Biggr)}\cdot\\
{\textstyle e^{-\sigma{\displaystyle \int_{0}^{\frac{1}{c\cos\theta}y-\frac{a_{2}}{c\cos\theta}}}\rho\left(\left|M\right|\right)\left(s+t-\frac{1}{c\cos\theta}y+\frac{a_{2}}{c\cos\theta},-sc\sin\theta+x+\frac{\sin\theta}{\cos\theta}y-\frac{\sin\theta}{\cos\theta}a_{2},sc\cos\theta+a_{2}\right)ds}}.\label{eq:olole-1-1-1-1-2-1}
\end{multline}

\begin{multline}
N_{3}\left(t,x,y\right)=N_{3}^{A}\left(t,x,y\right)\cdot\mathbb{I}_{\begin{cases}
x-ct\sin\theta\geq a_{1}\\
y+ct\cos\theta\leq b_{2}
\end{cases}}\left(t,x,y\right)+\\
N_{3}^{B}\left(t,x,y\right)\cdot\mathbb{I}_{\begin{cases}
x-ct\sin\theta\leq a_{1}\\
x\cos\theta+y\sin\theta\leq a_{1}\cos\theta+b_{2}\sin\theta
\end{cases}}\left(t,x,y\right)+\\
N_{3}^{C}\left(t,x,y\right)\cdot\mathbb{I}_{\begin{cases}
y+ct\cos\theta\geq b_{2}\\
x\cos\theta+y\sin\theta\geq a_{1}\cos\theta+b_{2}\sin\theta
\end{cases}}\left(t,x,y\right)\label{aoalal-1-1}
\end{multline}
where 
\begin{multline}
N_{3}^{A}\left(t,x,y\right)=\Biggl({\displaystyle \int_{0}^{t}}e^{\sigma\int_{0}^{s}\rho\left(\left|M\right|\right)\left(r,x+c\left(r-t\right)\sin\theta,y-c\left(r-t\right)\cos\theta\right)dr}\cdot Q_{3}^{\sigma}\left(\left|M\right|\right)\\
\left(s,x+c\left(s-t\right)\sin\theta,y-c\left(s-t\right)\cos\theta\right)ds+N_{3}^{0}\left(x-ct\sin\theta,y+ct\cos\theta\right)\Biggr)\cdot\\
e^{-\sigma{\textstyle {\displaystyle \int_{0}^{t}}}\rho\left(\left|M\right|\right)\left(s,x+c\left(s-t\right)\sin\theta,y-c\left(s-t\right)\cos\theta\right)ds}.\label{eq:olole-1-2}
\end{multline}
\begin{multline}
N_{3}^{B}\left(t,x,y\right)=\\
\Biggl({\textstyle {\displaystyle \int_{0}^{\frac{1}{c\sin\theta}x-\frac{a_{1}}{c\sin\theta}}}e^{\sigma\int_{0}^{s}\rho\left(\left|M\right|\right)\left(r+t-\frac{1}{c\sin\theta}x+\frac{a_{1}}{c\sin\theta},rc\sin\theta+a_{1},-rc\cos\theta+\frac{\cos\theta}{\sin\theta}x+y-\frac{\cos\theta}{\sin\theta}a_{1}\right)dr}}\cdot\\
{\textstyle Q_{3}^{\sigma}\left(\left|M\right|\right)\left(s+t-\frac{1}{c\sin\theta}x+\frac{a_{1}}{c\sin\theta},sc\sin\theta+a_{1},-sc\cos\theta+\frac{\cos\theta}{\sin\theta}x+y-\frac{\cos\theta}{\sin\theta}a_{1}\right)ds}+\\
{\textstyle N_{3}^{-}\left(t-\frac{1}{c\sin\theta}x+\frac{a_{1}}{c\sin\theta},\frac{\cos\theta}{\sin\theta}x+y-\frac{\cos\theta}{\sin\theta}a_{1}\right)\Biggr)}\cdot\\
{\textstyle e^{-\sigma{\displaystyle \int_{0}^{\frac{1}{c\sin\theta}x-\frac{a_{1}}{c\sin\theta}}}\rho\left(\left|M\right|\right)\left(s+t-\frac{1}{c\sin\theta}x+\frac{a_{1}}{c\sin\theta},sc\sin\theta+a_{1},-sc\cos\theta+\frac{\cos\theta}{\sin\theta}x+y-\frac{\cos\theta}{\sin\theta}a_{1}\right)ds}}.\label{eq:olole-1-1-1}
\end{multline}
\begin{multline}
N_{3}^{C}\left(t,x,y\right)=\\
\Biggl({\textstyle {\displaystyle \int_{0}^{-\frac{1}{c\cos\theta}y+\frac{b_{2}}{c\cos\theta}}}e^{\sigma\int_{0}^{s}\rho\left(\left|M\right|\right)\left(r+t+\frac{1}{c\cos\theta}y-\frac{b_{2}}{c\cos\theta},rc\sin\theta+x+\frac{\sin\theta}{\cos\theta}y-\frac{\sin\theta}{\cos\theta}b_{2},-rc\cos\theta+b_{2}\right)dr}}\cdot\\
{\textstyle Q_{3}^{\sigma}\left(\left|M\right|\right)\left(s+t+\frac{1}{c\cos\theta}y-\frac{b_{2}}{c\cos\theta},sc\sin\theta+x+\frac{\sin\theta}{\cos\theta}y-\frac{\sin\theta}{\cos\theta}b_{2},-sc\cos\theta+b_{2}\right)ds}+\\
{\textstyle N_{3}^{++}\left(t+\frac{1}{c\cos\theta}y-\frac{b_{2}}{c\cos\theta},x+\frac{\sin\theta}{\cos\theta}y-\frac{\sin\theta}{\cos\theta}b_{2}\right)\Biggr)}\cdot\\
{\textstyle e^{-\sigma{\displaystyle \int_{0}^{-\frac{1}{c\cos\theta}y+\frac{b_{2}}{c\cos\theta}}}\rho\left(\left|M\right|\right)\left(s+t+\frac{1}{c\cos\theta}y-\frac{b_{2}}{c\cos\theta},sc\sin\theta+x+\frac{\sin\theta}{\cos\theta}y-\frac{\sin\theta}{\cos\theta}b_{2},-sc\cos\theta+b_{2}\right)ds}}.\label{eq:olole-1-1-1-1-1}
\end{multline}

\begin{multline}
N_{4}\left(t,x,y\right)=N_{4}^{A}\left(t,x,y\right)\cdot\mathbb{I}_{\begin{cases}
x+ct\cos\theta\leq b_{1}\\
y+ct\sin\theta\leq b_{2}
\end{cases}}\left(t,x,y\right)+\\
N_{4}^{B}\left(t,x,y\right)\cdot\mathbb{I}_{\begin{cases}
x+ct\cos\theta\geq b_{1}\\
x\sin\theta-y\cos\theta\geq b_{1}\sin\theta-b_{2}\cos\theta
\end{cases}}\left(t,x,y\right)+\\
N_{4}^{C}\left(t,x,y\right)\cdot\mathbb{I}_{\begin{cases}
y+ct\sin\theta\geq b_{2}\\
x\sin\theta-y\cos\theta\leq b_{1}\sin\theta-b_{2}\cos\theta
\end{cases}}\left(t,x,y\right)\label{aoalal-1-1-1}
\end{multline}
where 
\begin{multline}
N_{4}^{A}\left(t,x,y\right)=\Biggl({\displaystyle \int_{0}^{t}}e^{\sigma\int_{0}^{s}\rho\left(\left|M\right|\right)\left(r,x-c\left(r-t\right)\cos\theta,y-c\left(r-t\right)\sin\theta\right)dr}\cdot Q_{4}^{\sigma}\left(\left|M\right|\right)\\
\left(s,x-c\left(s-t\right)\cos\theta,y-c\left(s-t\right)\sin\theta\right)ds+N_{4}^{0}\left(x+ct\cos\theta,y+ct\sin\theta\right)\Biggr)\cdot\\
e^{-\sigma{\textstyle {\displaystyle \int_{0}^{t}}}\rho\left(\left|M\right|\right)\left(s,x-c\left(s-t\right)\cos\theta,y-c\left(s-t\right)\sin\theta\right)ds}.\label{eq:olole-1-2-1}
\end{multline}
\begin{multline}
N_{4}^{B}\left(t,x,y\right)=\\
\Biggl({\textstyle {\displaystyle \int_{0}^{-\frac{1}{c\cos\theta}x+\frac{b_{1}}{c\cos\theta}}}e^{\sigma\int_{0}^{s}\rho\left(\left|M\right|\right)\left(r+t+\frac{1}{c\cos\theta}x-\frac{b_{1}}{c\cos\theta},-rc\cos\theta+b_{1},-rc\sin\theta-\frac{\sin\theta}{\cos\theta}x+y+\frac{\sin\theta}{\cos\theta}b_{1}\right)dr}}\cdot\\
{\textstyle Q_{4}^{\sigma}\left(\left|M\right|\right)\left(s+t+\frac{1}{c\cos\theta}x-\frac{b_{1}}{c\cos\theta},-sc\cos\theta+b_{1},-sc\sin\theta-\frac{\sin\theta}{\cos\theta}x+y+\frac{\sin\theta}{\cos\theta}b_{1}\right)ds}+\\
{\textstyle N_{4}^{+}\left(t+\frac{1}{c\cos\theta}x-\frac{b_{1}}{c\cos\theta},-\frac{\sin\theta}{\cos\theta}x+y+\frac{\sin\theta}{\cos\theta}b_{1}\right)\Biggr)}\cdot\\
{\textstyle e^{-\sigma{\displaystyle \int_{0}^{-\frac{1}{c\cos\theta}x+\frac{b_{1}}{c\cos\theta}}}\rho\left(\left|M\right|\right)\left(s+t+\frac{1}{c\cos\theta}x-\frac{b_{1}}{c\cos\theta},-sc\cos\theta+b_{1},-sc\sin\theta-\frac{\sin\theta}{\cos\theta}x+y+\frac{\sin\theta}{\cos\theta}b_{1}\right)ds}}.\label{eq:olole-1-1-1-2}
\end{multline}
\begin{multline}
N_{4}^{C}\left(t,x,y\right)=\\
\Biggl({\textstyle {\displaystyle \int_{0}^{-\frac{1}{c\sin\theta}y+\frac{b_{2}}{c\sin\theta}}}e^{\sigma\int_{0}^{s}\rho\left(\left|M\right|\right)\left(r+t+\frac{1}{c\sin\theta}y-\frac{b_{2}}{c\sin\theta},-rc\cos\theta+x-\frac{\cos\theta}{\sin\theta}y+\frac{\cos\theta}{\sin\theta}b_{2},-rc\sin\theta+b_{2}\right)dr}}\cdot\\
{\textstyle Q_{4}^{\sigma}\left(\left|M\right|\right)\left(s+t+\frac{1}{c\sin\theta}y-\frac{b_{2}}{c\sin\theta},-sc\cos\theta+x-\frac{\cos\theta}{\sin\theta}y+\frac{\cos\theta}{\sin\theta}b_{2},-sc\sin\theta+b_{2}\right)ds}+\\
{\textstyle N_{4}^{++}\left(t+\frac{1}{c\sin\theta}y-\frac{b_{2}}{c\sin\theta},x-\frac{\cos\theta}{\sin\theta}y+\frac{\cos\theta}{\sin\theta}b_{2}\right)\Biggr)}\cdot\\
{\textstyle e^{-\sigma{\displaystyle \int_{0}^{-\frac{1}{c\sin\theta}y+\frac{b_{2}}{c\sin\theta}}}\rho\left(\left|M\right|\right)\left(s+t+\frac{1}{c\sin\theta}y-\frac{b_{2}}{c\sin\theta},-sc\cos\theta+x-\frac{\cos\theta}{\sin\theta}y+\frac{\cos\theta}{\sin\theta}b_{2},-sc\sin\theta+b_{2}\right)ds}}.\label{eq:olole-1-1-1-1-1-2}
\end{multline}

Now we have 

\[
Q_{1}^{\sigma}\left(\left|M\right|\right)=\sigma\left(\left|M_{1}\right|+\left|M_{2}\right|+\left|M_{3}\right|\right)\left|M_{1}\right|+2cS\left|M_{2}\right|\left|M_{3}\right|+\left(\sigma-2cS\right)\left|M_{1}\right|\left|M_{4}\right|
\]
 
\[
Q_{2}^{\sigma}\left(\left|M\right|\right)=\sigma\left(\left|M_{1}\right|+\left|M_{2}\right|+\left|M_{4}\right|\right)\left|M_{2}\right|+2cS\left|M_{1}\right|\left|M_{4}\right|+\left(\sigma-2cS\right)\left|M_{2}\right|\left|M_{3}\right|
\]

\[
Q_{3}^{\sigma}\left(\left|M\right|\right)=\sigma\left(\left|M_{1}\right|+\left|M_{3}\right|+\left|M_{4}\right|\right)\left|M_{3}\right|+2cS\left|M_{1}\right|\left|M_{4}\right|+\left(\sigma-2cS\right)\left|M_{2}\right|\left|M_{3}\right|
\]

\[
Q_{4}^{\sigma}\left(\left|M\right|\right)=\sigma\left(\left|M_{2}\right|+\left|M_{3}\right|+\left|M_{4}\right|\right)\left|M_{4}\right|+2cS\left|M_{2}\right|\left|M_{3}\right|+\left(\sigma-2cS\right)\left|M_{1}\right|\left|M_{4}\right|.
\]

and we conclude that for $\sigma\geq2cS,$ the solution $N=\left(N_{1},N_{2},N_{3},N_{4}\right)$
of $\left(\Sigma_{\sigma,M}\right)$ is non-negative.

Let us consider the operator $\mathcal{T}^{\sigma}:M\longrightarrow N_{M}$
where $N_{M}$ is the unique non-negative solution of the problem
$\left(\Sigma_{\sigma,M}\right)$ for sufficiently large $\sigma.$

We easily verify from the statements of $\left(\Sigma_{\sigma}\right)$
(\ref{eq:ssdffz-1})-(\ref{eq:qqqqsa-1}) and $\left(\Sigma_{\sigma,M}\right)$
(\ref{eq:ssdffz-1-1})-(\ref{eq:qqqqsa-1-1}) that $\mathcal{T}^{\sigma}\left(M\right)=M$
iff $M$ is a solution of $\left(\Sigma_{\sigma}\right)$ for sufficiently
large $\sigma$ i.e. $M$ is a solution of $\Sigma$ as $\left(\Sigma_{\sigma}\right)$
is equivalent to $\Sigma.$ As $\mathcal{T}^{\sigma}\left(M\right)$
is non-negative for sufficiently large $\sigma$ , so is any solution
$M$ of $\Sigma.$ 
\end{proof}

\section{Existence and uniqueness \label{ooolsppikksi}}

\subsection{Auxilliary operator\label{subsec:Fixed-point-problem}}

Let $M=\left(M_{1},M_{2},M_{3},M_{4}\right)$ be fixed . Let us replace
$N$ in the second member of $\Sigma$ (\ref{eq:erter})-(\ref{eq:lsoo})
by $M$. We obtain the linear system $\left(\Sigma_{M}\right)$ defined
by
\begin{align}
{\textstyle \dfrac{\p N_{1}}{\p t}+c\cos\theta\dfrac{\p N_{1}}{\p x}} & {\displaystyle +c\sin\theta\dfrac{\p N_{1}}{\p y}}=Q(M),\;\left(t,x,y\right)\in\mathring{\mathscr{P}}\label{eq:ssdffz-1-1-1}\\
\dfrac{\p N_{2}}{\p t}-c\sin\theta\dfrac{\p N_{2}}{\p x} & +c\cos\theta\dfrac{\p N_{2}}{\p y}=-Q(M),\;\left(t,x,y\right)\in\mathring{\mathscr{P}}\\
\dfrac{\p N_{3}}{\p t}+c\sin\theta\dfrac{\p N_{3}}{\p x} & -c\cos\theta\dfrac{\p N_{3}}{\p y}=-Q(M),\;\left(t,x,y\right)\in\mathring{\mathscr{P}}\\
\dfrac{\p N_{4}}{\p t}-c\cos\theta\dfrac{\p N_{4}}{\p x} & -c\sin\theta\dfrac{\p N_{4}}{\p y}=Q(M),\;\left(t,x,y\right)\in\mathring{\mathscr{P}}\label{eq:qqqqsa-1-1-1}
\end{align}
with the conditions \ref{eq:lsos}-\ref{eq:mqppa}. $\left(\Sigma_{M}\right)$
has an unique continuous solution $\mathcal{T}\left(M\right)=\left(\mathcal{T}_{i}\left(M\right)\right)_{i=1}^{4}$
defined in the appendix (\ref{sec:Solution-of-the}) by expressions
(\ref{eq:oloooso})-(\ref{eq:olole-1-1-1-1-1-2-1})

It is immediate that the solutions of $\Sigma$ are the fixed points
of the operator
\begin{align}
\mathcal{T}:C\left(\begin{array}{r}
\mathscr{P}\end{array};\R^{4}\right) & \longrightarrow C\left(\begin{array}{r}
\mathscr{P}\end{array};\R^{4}\right)\nonumber \\
M & \longmapsto\mathcal{T}\left(M\right)=\left(\mathcal{T}_{i}\left(M\right)\right)_{i=1}^{4}.\label{eq:kozep}
\end{align}

We easily prove as in \cite{key-1}, that 
\begin{lem}
\label{lem:-is-continuous.}$\mathcal{T}$ is continuous.
\end{lem}

\begin{proof}
\ref{eq:oloooso}-\ref{eq:looaoolzo} give for $M,N\in C\left(\mathscr{P},\R^{4}\right):$
\begin{multline}
\left\Vert \mathcal{T}_{1}\left(M\right)-\mathcal{T}_{1}\left(N\right)\right\Vert _{\infty}\leq\\
\max\Biggl\{\sup_{\left(t,x,y\right)\in\mathscr{P}}\Biggl|\int_{0}^{t}\left[Q\left(M\right)-Q\left(N\right)\right]\left(s,x+c\left(s-t\right)\cos\theta,y+c\left(s-t\right)\sin\theta\right)ds\Biggr|;\\
\sup_{\left(t,x,y\right)\in\mathscr{P}}\Biggl|\int_{0}^{\frac{1}{c\cos\theta}x-\frac{a_{1}}{c\cos\theta}}\left[Q\left(M\right)-Q\left(N\right)\right]\\
\left(s+t-\frac{1}{c\cos\theta}x+\frac{a_{1}}{c\cos\theta},sc\cos\theta+a_{1},sc\sin\theta-\frac{\sin\theta}{\cos\theta}x+y+\frac{\sin\theta}{\cos\theta}a_{1}\right)ds\Biggr|;\\
\sup_{\left(t,x,y\right)\in\mathscr{P}}\Biggl|\int_{0}^{\frac{1}{c\sin\theta}y-\frac{a_{2}}{c\sin\theta}}\left[Q\left(M\right)-Q\left(N\right)\right]\\
\left(s+t-\frac{1}{c\sin\theta}y+\frac{a_{2}}{c\sin\theta},sc\cos\theta+x-\frac{\cos\theta}{\sin\theta}y+\frac{\cos\theta}{\sin\theta}a_{2},sc\sin\theta+a_{2}\right)ds\Biggr|\Biggl\}.\label{aoalal-2-1-1-1-1-1-1}
\end{multline}
But 
\begin{multline}
Q\left(M\right)-Q\left(N\right)=2cS\left(M_{2}-N_{2}\right)M_{3}+2cSN_{2}\left(M_{3}-N_{3}\right)\\
-2cS\left(M_{1}-N_{1}\right)M_{4}-2cSN_{1}\left(M_{4}-N_{4}\right),
\end{multline}
hence 
\begin{multline}
\left\Vert Q\left(M\right)-Q\left(N\right)\right\Vert _{\infty}\leq2cS\left\Vert M_{2}-N_{2}\right\Vert _{\infty}\left\Vert M_{3}\right\Vert _{\infty}+2cS\left\Vert N_{2}\right\Vert _{\infty}\left\Vert M_{3}-N_{3}\right\Vert _{\infty}\\
+2cS\left\Vert M_{1}-N_{1}\right\Vert _{\infty}\left\Vert M_{4}\right\Vert _{\infty}+2cS\left\Vert N_{1}\right\Vert _{\infty}\left\Vert M_{4}-N_{4}\right\Vert _{\infty},
\end{multline}
and 
\begin{multline}
\left\Vert Q\left(M\right)-Q\left(N\right)\right\Vert _{\infty}\leq4cS\left\Vert M-N\right\Vert \left\Vert M\right\Vert +4cS\left\Vert N\right\Vert \left\Vert M-N\right\Vert \\
\left\Vert Q\left(M\right)-Q\left(N\right)\right\Vert _{\infty}\leq4cS\left(\left\Vert M\right\Vert +\left\Vert N\right\Vert \right)\left\Vert M-N\right\Vert ;
\end{multline}

hence \ref{aoalal-2-1-1-1-1-1-1} implies 
\begin{multline}
\left\Vert \mathcal{T}_{1}\left(M\right)-\mathcal{T}_{1}\left(N\right)\right\Vert _{\infty}\leq\\
\max\Biggl\{ T,\dfrac{b_{1}-a_{1}}{c\cos\theta},{\displaystyle \frac{b_{2}-a_{2}}{c\sin\theta}}\Biggl\}\cdot4cS\left(\left\Vert M\right\Vert +\left\Vert N\right\Vert \right)\left\Vert M-N\right\Vert .\label{aoalal-2-1-1-1-1-1-1-1-1-2}
\end{multline}
Similar inequalities hold: from \ref{eq:ppmms-2-1}-\ref{eq:opos}
we have
\begin{multline}
\left\Vert \mathcal{T}_{2}\left(M\right)-\mathcal{T}_{2}\left(N\right)\right\Vert _{\infty}\leq\\
\max\Biggl\{ T,\dfrac{b_{1}-a_{1}}{c\sin\theta},{\displaystyle \frac{b_{2}-a_{2}}{c\cos\theta}}\Biggl\}\cdot4cS\left(\left\Vert M\right\Vert +\left\Vert N\right\Vert \right)\left\Vert M-N\right\Vert ;\label{aoalal-2-1-1-1-1-1-1-1-1-1-2}
\end{multline}
from \ref{aoalal-1-1-2}-\ref{eq:loiuij},
\begin{multline}
\left\Vert \mathcal{T}_{3}\left(M\right)-\mathcal{T}_{3}\left(N\right)\right\Vert _{\infty}\leq\\
\max\Biggl\{ T,\dfrac{b_{1}-a_{1}}{c\sin\theta},{\displaystyle \frac{b_{2}-a_{2}}{c\cos\theta}}\Biggl\}\cdot4cS\left(\left\Vert M\right\Vert +\left\Vert N\right\Vert \right)\left\Vert M-N\right\Vert .\label{aoalal-2-1-1-1-1-1-1-1-1-1-1-2}
\end{multline}
and from (\ref{aoalal-1-1-1-2})-(\ref{eq:olole-1-1-1-1-1-2-1}) 
\begin{multline}
\left\Vert \mathcal{T}_{4}\left(M\right)-\mathcal{T}_{4}\left(N\right)\right\Vert _{\infty}\leq\\
\max\Biggl\{ T,\dfrac{b_{1}-a_{1}}{c\cos\theta},{\displaystyle \frac{b_{2}-a_{2}}{c\sin\theta}}\Biggl\}\cdot4cS\left(\left\Vert M\right\Vert +\left\Vert N\right\Vert \right)\left\Vert M-N\right\Vert .\label{aoalal-2-1-1-1-1-1-1-1-1-1-1-1-1}
\end{multline}
 Now from (\ref{aoalal-2-1-1-1-1-1-1-1-1-2}),(\ref{aoalal-2-1-1-1-1-1-1-1-1-1-2}),
(\ref{aoalal-2-1-1-1-1-1-1-1-1-1-1-2}),(\ref{aoalal-2-1-1-1-1-1-1-1-1-1-1-1-1})
and

$\left\Vert \mathcal{T}\left(M\right)-\mathcal{T}\left(N\right)\right\Vert ={\displaystyle \max_{1\leq i\leq4}}\left\Vert \mathcal{T}_{i}\left(M\right)-\mathcal{T}_{i}\left(N\right)\right\Vert _{\infty}$we
have 
\begin{multline}
\left\Vert \mathcal{T}\left(M\right)-\mathcal{T}\left(N\right)\right\Vert \leq\\
\underbrace{\max\Biggl\{ T,\dfrac{b_{1}-a_{1}}{c\cos\theta},{\displaystyle \frac{b_{2}-a_{2}}{c\sin\theta}},\dfrac{b_{1}-a_{1}}{c\sin\theta},{\displaystyle \frac{b_{2}-a_{2}}{c\cos\theta}}\Biggl\}\cdot4cS}_{\equiv p'}\left(\left\Vert M\right\Vert +\left\Vert N\right\Vert \right)\left\Vert M-N\right\Vert .\label{eq:lsoopa-2}
\end{multline}

We deduce that $\mathcal{T}$ is continuous.
\end{proof}
\begin{prop}
\label{prop:Suppose--such}Suppose $M=\left(M_{1},M_{2},M_{3},M_{4}\right)\in C\left(\begin{array}{r}
\mathscr{P}\end{array};\R^{4}\right)$ such that $\dfrac{\partial M_{i}}{\partial t},\dfrac{\partial M_{i}}{\partial x},\dfrac{\partial M_{i}}{\partial y}$
are defined in $\mathring{\mathscr{P}},$ except possibly on a finite
number of planes, and are continuous and bounded forall $i=1,2,3,4.$
Then all the derivatives $\dfrac{\partial\mathcal{T}_{i}\left(M\right)}{\partial t},\dfrac{\partial\mathcal{T}_{i}\left(M\right)}{\partial x},\dfrac{\partial\mathcal{T}_{i}\left(M\right)}{\partial y}$
$\left(i=1,2,3,4\right)$ are defined in $\mathring{\mathscr{P}},$
except possibly on a finite number of planes, and are continuous and
bounded. 

In other words, if we denote by $\mathscr{E}$ the sub-space of $C\left(\begin{array}{r}
\mathscr{P}\end{array};\R\right)$ consisting of functions $u$ that are continuous on $\mathscr{P}$
such that $\dfrac{\partial u}{\partial t},\dfrac{\partial u}{\partial x},\dfrac{\partial u}{\partial y}$
are defined in $\mathring{\mathscr{P}}$ except possibly on a finite
number of planes, and are continuous and bounded, then $\mathcal{T}\left(\mathscr{E}^{4}\right)\subset\mathscr{E}^{4}.$
\end{prop}

\begin{proof}
This follows from the explicit formula of the derivatives of $\mathcal{T}_{i}\left(M\right)$
$\left(i=1,2,3,4\right)$ given in the appendix (\ref{sec:Derivatives})
\end{proof}
For $N=\left(N_{i}\right)_{i=1}^{4}\in\mathscr{E}^{4},$ let $\mathscr{V}\left(N\right)\equiv\max\left\{ \left\Vert N\right\Vert ,\left\Vert \dfrac{\partial N}{\partial t}\right\Vert ,\left\Vert \dfrac{\partial N}{\partial x}\right\Vert ,\left\Vert \dfrac{\partial N}{\partial y}\right\Vert \right\} $
and consider for $R>0,$ $\mathscr{B}_{R}\equiv\left\{ N\in\mathscr{E}^{4}:\mathscr{V}\left(N\right)\leq R\right\} .$

\begin{prop}
\label{prop:The-sets-}The sets $\mathscr{B}_{R},\left(R>0\right)$
are non-empty convex subsets of $C\left(\begin{array}{r}
\mathscr{P}\end{array};\R^{4}\right)$ and are relatively compact in $\left(C\left(\begin{array}{r}
\mathscr{P}\end{array};\R^{4}\right),\left\Vert \cdot\right\Vert \right).$ 
\end{prop}

\begin{proof}
$\mathscr{B}_{R}$ contains the zero function. $\forall M,N\in\mathscr{B}_{R},$
and $\forall\lambda\in\R$ such that $0\leq\lambda\leq1$ we have
$\lambda M+\left(1-\lambda\right)N\in\mathscr{E}^{4};$ $\mathscr{V}\left(\lambda M+\left(1-\lambda\right)N\right)\leq R$
is immediate. $\mathscr{B}_{R}$ is thus convex.

By definition, $\mathscr{B}_{R}$ is a bounded set in $\left(C\left(\begin{array}{r}
\mathscr{P}\end{array};\R^{4}\right),\left\Vert \cdot\right\Vert \right)$ and the real functions $\dfrac{\partial M_{i}}{\partial t},\dfrac{\partial M_{i}}{\partial x},\dfrac{\partial M_{i}}{\partial y},\left(i=1,2,3,4\right)$
are uniformly bounded and continuous on their domain for $M\in\mathscr{B}_{R}.$
Thus there exists a constant $k^{i}>0$ independant of $M$ such that
$\forall M\in\mathcal{\mathscr{B}}_{R},\forall\left(t,x,y\right)$
in the domain of $\left(\dfrac{\partial M_{i}}{\partial t},\dfrac{\partial M_{i}}{\partial x},\dfrac{\partial M_{i}}{\partial y}\right),$
it holds
\begin{equation}
\left\Vert d\left(M_{i}\right)\left(t,x,y\right)\right\Vert _{\mathscr{L}\left(\R^{3},\R\right)}\leq k^{i}.\label{eq:oloozor}
\end{equation}

Let $M\in\mathcal{\mathscr{B}}_{R}$ and $\left(t,x,y\right),\left(t',x',y'\right)\in\mathscr{P}$
such that $\left(\dfrac{\partial M}{\partial t},\dfrac{\partial M}{\partial x},\dfrac{\partial M}{\partial y}\right)$
is defined on the segment

$\left[\left(t,x,y\right),\left(t',x',y'\right)\right]\equiv\left\{ \left(t,x,y\right)+\alpha\left(t'-t,x'-x,y'-y\right)\vert\left(0\leq\alpha\leq1\right)\right\} $.
As $\mathscr{P}$ is a convex set, then by the mean value inequality,
inequality (\ref{eq:oloozor}) implies that forall $i=1,2,3,4,$ 
\begin{align}
\left|M_{i}\left(t,x,y\right)-M_{i}\left(t',x',y'\right)\right| & \leq k^{i}\left\Vert \left(t,x,y\right)-\left(t',x',y'\right)\right\Vert _{\R^{3}}.\label{eq:olopposn}
\end{align}

For any $\varepsilon>0$, let ${\displaystyle \alpha_{\varepsilon}\equiv\min_{1\leq i\leq4}\dfrac{\varepsilon}{k^{i}}}.$
It follows from (\ref{eq:olopposn}) that 

\begin{multline}
\left\Vert \left(t,x,y\right)-\left(t',x',y'\right)\right\Vert _{\R^{3}}<\alpha_{\varepsilon}\implies\\
\left\Vert M\left(t,x,y\right)-M\left(t',x',y'\right)\right\Vert _{\R^{4}}\leq\varepsilon\label{eq:ayueacw-1}
\end{multline}

The points where $\left(\dfrac{\partial M}{\partial t},\dfrac{\partial M}{\partial x},\dfrac{\partial M}{\partial y}\right)$
is not defined are in planes. Therefore as $M$ is continuous on $\mathscr{P,}$
we can extend (\ref{eq:ayueacw-1}) to $\mathscr{P}$. We have 
\begin{multline}
\forall\varepsilon>0,\exists\alpha_{\varepsilon}>0,\forall M\in\mathscr{B}_{R},\forall\left(t,x,y\right),\left(t',x',y'\right)\in\mathscr{P}:\\
\left\Vert \left(t,x,y\right)-\left(t',x',y'\right)\right\Vert <\alpha_{\varepsilon}\implies\\
\left\Vert M\left(t,x,y\right)-M\left(t',x',y'\right)\right\Vert _{\R^{4}}\leq\varepsilon
\end{multline}

which means that $\mathscr{B}_{R}$ is equicontinuous in $\left(C\left(\begin{array}{r}
\mathscr{P}\end{array};\R^{4}\right),\left\Vert \cdot\right\Vert \right).$We conclude by Arzelà-Ascoli theorem that $\mathscr{B}_{R}$ is relatively
compact in

$\left(C\left(\begin{array}{r}
\mathscr{P}\end{array};\R^{4}\right),\left\Vert \cdot\right\Vert \right).$ 
\end{proof}
\begin{prop}
\label{prop::opps}There exists $p,q>0$ such that 

\begin{equation}
\forall R>0,\mathcal{T}\left(\mathscr{B}_{R}\right)\subset\mathscr{B}_{pR^{2}+q}.
\end{equation}
\end{prop}

\begin{proof}
By proposition (\ref{prop:Suppose--such}), if $M\in\mathscr{B}_{R}\subset\mathscr{E}^{4}$
then $\mathcal{T}\left(M\right)\in\mathscr{E}^{4}.$ For any $M\in\mathscr{E}^{4},$
$Q\left(M\right)={\displaystyle 2cS\left(M_{2}M_{3}-M_{1}M_{4}\right)}$
implies 
\begin{equation}
\left\Vert Q\left(M\right)\right\Vert _{\infty}\leq4cS\left(\mathscr{V}\left(M\right)\right)^{2}.\label{eq:kooll}
\end{equation}
We then have from the expressions \ref{eq:oloooso}-\ref{eq:olole-1-1-1-1-1-2-1}
of the solution of the system $\left(\Sigma_{M}\right)$, the following
inequalities for $M\in\mathscr{E}^{4}:$ 
\begin{multline}
\left\Vert \mathcal{T}_{1}\left(M\right)\right\Vert _{\infty}\leq4cS\cdot\max\Biggl\{ T,\dfrac{b_{1}-a_{1}}{c\cos\theta},{\displaystyle \frac{b_{2}-a_{2}}{c\sin\theta}}\Biggl\}\left(\mathscr{V}\left(M\right)\right)^{2}+\\
\max\Biggl\{\left\Vert N_{1}^{0}\right\Vert _{1},\left\Vert N_{1}^{-}\right\Vert _{1},\left\Vert N_{1}^{--}\right\Vert _{1}\Biggr\}.\label{eq:ksoloa}
\end{multline}

\begin{multline}
\left\Vert \mathcal{T}_{2}\left(M\right)\right\Vert _{\infty}\leq4cS\cdot\max\Biggl\{ T,\dfrac{b_{1}-a_{1}}{c\sin\theta},{\displaystyle \frac{b_{2}-a_{2}}{c\cos\theta}}\Biggl\}\left(\mathscr{V}\left(M\right)\right)^{2}+\\
\max\Biggl\{\left\Vert N_{2}^{0}\right\Vert _{1},\left\Vert N_{2}^{+}\right\Vert _{1},\left\Vert N_{2}^{--}\right\Vert _{1}\Biggr\}\label{eq:ksoloa-1}
\end{multline}

\begin{multline}
\left\Vert \mathcal{T}_{3}\left(M\right)\right\Vert _{\infty}\leq4cS\cdot\max\Biggl\{ T,\dfrac{b_{1}-a_{1}}{c\sin\theta},{\displaystyle \frac{b_{2}-a_{2}}{c\cos\theta}}\Biggl\}\left(\mathscr{V}\left(M\right)\right)^{2}+\\
\max\Biggl\{\left\Vert N_{3}^{0}\right\Vert _{1},\left\Vert N_{3}^{-}\right\Vert _{1},\left\Vert N_{3}^{++}\right\Vert _{1}\Biggr\}\label{eq:ksoloa-1-1}
\end{multline}

\begin{multline}
\left\Vert \mathcal{T}_{4}\left(M\right)\right\Vert _{\infty}\leq4cS\cdot\max\Biggl\{ T,\dfrac{b_{1}-a_{1}}{c\cos\theta},{\displaystyle \frac{b_{2}-a_{2}}{c\sin\theta}}\Biggl\}\left(\mathscr{V}\left(M\right)\right)^{2}+\\
\max\Biggl\{\left\Vert N_{4}^{0}\right\Vert _{1},\left\Vert N_{4}^{+}\right\Vert _{1},\left\Vert N_{4}^{++}\right\Vert _{1}\Biggr\}.\label{eq:ksoloa-1-1-1}
\end{multline}

and 
\begin{multline}
\left\Vert \mathcal{T}\left(M\right)\right\Vert \leq4cS\cdot\max\Biggl\{ T,\dfrac{b_{1}-a_{1}}{c\cos\theta},{\displaystyle \frac{b_{2}-a_{2}}{c\sin\theta}},\dfrac{b_{1}-a_{1}}{c\sin\theta},{\displaystyle \frac{b_{2}-a_{2}}{c\cos\theta}}\Biggl\}\left(\mathscr{V}\left(M\right)\right)^{2}+\\
\max_{1\leq i\leq4}\Biggl\{\left\Vert N_{i}^{0}\right\Vert _{1},\left\Vert N_{1}^{-}\right\Vert _{1},\left\Vert N_{1}^{--}\right\Vert _{1},\left\Vert N_{2}^{+}\right\Vert _{1},\left\Vert N_{2}^{--}\right\Vert _{1},\\
\left\Vert N_{3}^{-}\right\Vert _{1},\left\Vert N_{3}^{++}\right\Vert _{1},\left\Vert N_{4}^{+}\right\Vert _{1},\left\Vert N_{4}^{++}\right\Vert _{1}\Biggr\}.\label{eq:lospp}
\end{multline}

We have ${\displaystyle \dfrac{\partial Q\left(M\right)}{\partial t}=2cS\left(\dfrac{\partial M_{2}}{\partial t}M_{3}+M_{2}\dfrac{\partial M_{3}}{\partial t}-\dfrac{\partial M_{1}}{\partial t}M_{4}-M_{1}\dfrac{\partial M_{4}}{\partial t}\right)}$
and similar expressions for $\dfrac{\partial Q\left(M\right)}{\partial x},\dfrac{\partial Q\left(M\right)}{\partial y}.$
Therefore 

\begin{multline}
\left\Vert \dfrac{\partial Q\left(M\right)}{\partial t}\right\Vert _{\infty},\left\Vert \dfrac{\partial Q\left(M\right)}{\partial x}\right\Vert _{\infty},\left\Vert \dfrac{\partial Q\left(M\right)}{\partial y}\right\Vert _{\infty}\leq\\
2cS\cdot4\left(\mathscr{V}\left(M\right)\right)^{2}\leq8cS\left(\mathscr{V}\left(M\right)\right)^{2}.\label{eq:lopoi}
\end{multline}

(\ref{eq:lopoi}) with explicit formula of the derivatives

$\dfrac{\partial\mathcal{T}_{i}\left(M\right)}{\partial t},\dfrac{\partial\mathcal{T}_{i}\left(M\right)}{\partial x},\dfrac{\partial\mathcal{T}_{i}\left(M\right)}{\partial y},\left(i=1,2,3,4\right)$
in the appendix (\ref{sec:Derivatives}) imply:
\begin{multline}
\left\Vert \dfrac{\partial\mathcal{T}_{1}\left(M\right)}{\partial t}\right\Vert _{\infty}\leq4cS\max\left\{ 1+2T\left(c\cos\theta+c\sin\theta\right);2\dfrac{b_{1}-a_{1}}{c\cos\theta};2{\displaystyle \frac{b_{2}-a_{2}}{c\sin\theta}}\right\} \left(\mathscr{V}\left(M\right)\right)^{2}\\
\max\Biggl\{\left(c\cos\theta+c\sin\theta\right)\left\Vert N_{1}^{0}\right\Vert _{1};\left\Vert N_{1}^{-}\right\Vert _{1};\left\Vert N_{1}^{--}\right\Vert _{1}\Biggr\};\label{eq:lopiok}
\end{multline}
\begin{multline}
\left\Vert \dfrac{\partial\mathcal{T}_{2}\left(M\right)}{\partial t}\right\Vert _{\infty}\leq4cS\max\left\{ 1+2T\left(c\cos\theta+c\sin\theta\right);2\dfrac{b_{1}-a_{1}}{c\sin\theta};2{\displaystyle \frac{b_{2}-a_{2}}{c\cos\theta}}\right\} \left(\mathscr{V}\left(M\right)\right)^{2}\\
\max\Biggl\{\left(c\cos\theta+c\sin\theta\right)\left\Vert N_{2}^{0}\right\Vert _{1};\left\Vert N_{2}^{+}\right\Vert _{1};\left\Vert N_{2}^{--}\right\Vert _{1}\Biggr\};\label{eq:sloos}
\end{multline}

\begin{multline}
\left\Vert \dfrac{\partial\mathcal{T}_{3}\left(M\right)}{\partial t}\right\Vert _{\infty}\leq4cS\max\left\{ 1+2T\left(c\cos\theta+c\sin\theta\right);2\dfrac{b_{1}-a_{1}}{c\sin\theta};2{\displaystyle \frac{b_{2}-a_{2}}{c\cos\theta}}\right\} \left(\mathscr{V}\left(M\right)\right)^{2}\\
\max\Biggl\{\left(c\cos\theta+c\sin\theta\right)\left\Vert N_{3}^{0}\right\Vert _{1};\left\Vert N_{3}^{-}\right\Vert _{1};\left\Vert N_{3}^{++}\right\Vert _{1}\Biggr\};\label{eq:qmpql}
\end{multline}

\begin{multline}
\left\Vert \dfrac{\partial\mathcal{T}_{4}\left(M\right)}{\partial t}\right\Vert _{\infty}\leq4cS\max\left\{ 1+2T\left(c\cos\theta+c\sin\theta\right);2\dfrac{b_{1}-a_{1}}{c\cos\theta};2{\displaystyle \frac{b_{2}-a_{2}}{c\sin\theta}}\right\} \left(\mathscr{V}\left(M\right)\right)^{2}\\
\max\Biggl\{\left(c\cos\theta+c\sin\theta\right)\left\Vert N_{4}^{0}\right\Vert _{1};\left\Vert N_{4}^{+}\right\Vert _{1};\left\Vert N_{4}^{++}\right\Vert _{1}\Biggr\};\label{eq:qmpql-1}
\end{multline}
\begin{multline}
\left\Vert \dfrac{\partial\mathcal{T}_{1}\left(M\right)}{\partial x}\right\Vert _{\infty}\leq\\
4cS\max\left\{ 2T;\dfrac{1}{c\cos\theta}+2\dfrac{b_{1}-a_{1}}{c\cos\theta}\left(\dfrac{1}{c\cos\theta}+\dfrac{\sin\theta}{\cos\theta}\right);2{\displaystyle \frac{b_{2}-a_{2}}{c\sin\theta}}\right\} \left(\mathscr{V}\left(M\right)\right)^{2}\\
\max\Biggl\{\left\Vert N_{1}^{0}\right\Vert _{1};\left(\dfrac{1}{c\cos\theta}+\dfrac{\sin\theta}{\cos\theta}\right)\left\Vert N_{1}^{-}\right\Vert _{1};\left\Vert N_{1}^{--}\right\Vert _{1}\Biggr\};\label{eq:qmpql-1-1}
\end{multline}
\begin{multline}
\left\Vert \dfrac{\partial\mathcal{T}_{2}\left(M\right)}{\partial x}\right\Vert _{\infty}\leq\\
4cS\max\left\{ 2T;\dfrac{1}{c\sin\theta}+2\dfrac{b_{1}-a_{1}}{c\sin\theta}\left(\dfrac{1}{c\sin\theta}+\dfrac{\cos\theta}{\sin\theta}\right);2{\displaystyle \frac{b_{2}-a_{2}}{c\cos\theta}}\right\} \left(\mathscr{V}\left(M\right)\right)^{2}\\
\max\Biggl\{\left\Vert N_{2}^{0}\right\Vert _{1};\left(\dfrac{1}{c\sin\theta}+\dfrac{\cos\theta}{\sin\theta}\right)\left\Vert N_{2}^{+}\right\Vert _{1};\left\Vert N_{2}^{--}\right\Vert _{1}\Biggr\};\label{eq:qmpql-1-1-1}
\end{multline}
\begin{multline}
\left\Vert \dfrac{\partial\mathcal{T}_{3}\left(M\right)}{\partial x}\right\Vert _{\infty}\leq\\
4cS\max\left\{ 2T;\dfrac{1}{c\sin\theta}+2\dfrac{b_{1}-a_{1}}{c\sin\theta}\left(\dfrac{1}{c\sin\theta}+\dfrac{\cos\theta}{\sin\theta}\right);2{\displaystyle \frac{b_{2}-a_{2}}{c\cos\theta}}\right\} \left(\mathscr{V}\left(M\right)\right)^{2}\\
\max\Biggl\{\left\Vert N_{3}^{0}\right\Vert _{1};\left(\dfrac{1}{c\sin\theta}+\dfrac{\cos\theta}{\sin\theta}\right)\left\Vert N_{3}^{-}\right\Vert _{1};\left\Vert N_{3}^{++}\right\Vert _{1}\Biggr\};\label{eq:qmpql-1-1-1-1}
\end{multline}

\begin{multline}
\left\Vert \dfrac{\partial\mathcal{T}_{4}\left(M\right)}{\partial x}\right\Vert _{\infty}\leq\\
4cS\max\left\{ 2T;\dfrac{1}{c\cos\theta}+2\dfrac{b_{1}-a_{1}}{c\cos\theta}\left(\dfrac{1}{c\cos\theta}+\dfrac{\sin\theta}{\cos\theta}\right);2{\displaystyle \frac{b_{2}-a_{2}}{c\sin\theta}}\right\} \left(\mathscr{V}\left(M\right)\right)^{2}\\
\max\Biggl\{\left\Vert N_{4}^{0}\right\Vert _{1};\left(\dfrac{1}{c\cos\theta}+\dfrac{\sin\theta}{\cos\theta}\right)\left\Vert N_{4}^{+}\right\Vert _{1};\left\Vert N_{4}^{++}\right\Vert _{1}\Biggr\}.\label{eq:koalla}
\end{multline}
\ref{eq:lopiok}-\ref{eq:qmpql-1} imply
\begin{multline}
\left\Vert \dfrac{\partial\mathcal{T}\left(M\right)}{\partial t}\right\Vert \leq\\
4cS\max\biggl\{1+2T\left(c\cos\theta+c\sin\theta\right);2\max\left\{ \dfrac{1}{c\cos\theta};\frac{1}{c\sin\theta}\right\} \left(b_{1}-a_{1}\right);\\
2\max\left\{ \dfrac{1}{c\cos\theta};\frac{1}{c\sin\theta}\right\} \left(b_{2}-a_{2}\right)\biggr\}\left(\mathscr{V}\left(M\right)\right)^{2}+\\
\max_{1\leq i\leq4}\Biggl\{\left(c\cos\theta+c\sin\theta\right)\left\Vert N_{i}^{0}\right\Vert _{1},\left\Vert N_{1}^{-}\right\Vert _{1},\left\Vert N_{1}^{--}\right\Vert _{1},\\
\left\Vert N_{2}^{+}\right\Vert _{1},\left\Vert N_{2}^{--}\right\Vert _{1},\left\Vert N_{3}^{-}\right\Vert _{1},\left\Vert N_{3}^{++}\right\Vert _{1},\left\Vert N_{4}^{+}\right\Vert _{1},\left\Vert N_{4}^{++}\right\Vert _{1}\Biggr\};\label{eq:qmpql-1-2}
\end{multline}
\ref{eq:qmpql-1-1}-\ref{eq:koalla} imply
\begin{multline}
\left\Vert \dfrac{\partial\mathcal{T}\left(M\right)}{\partial x}\right\Vert \leq\\
4cS\max\Biggl\{2T;\max\left\{ \dfrac{1}{c\cos\theta};\frac{1}{c\sin\theta}\right\} +2\left(b_{1}-a_{1}\right)\max\biggl\{\dfrac{1}{c\cos\theta}\left(\dfrac{1}{c\cos\theta}+\dfrac{\sin\theta}{\cos\theta}\right);\\
\dfrac{1}{c\sin\theta}\left(\dfrac{1}{c\sin\theta}+\dfrac{\cos\theta}{\sin\theta}\right)\biggr\};2\left(b_{2}-a_{2}\right)\max\left\{ \dfrac{1}{c\cos\theta};\frac{1}{c\sin\theta}\right\} \Biggr\}\left(\mathscr{V}\left(M\right)\right)^{2}+\\
\max_{1\leq i\leq4}\Biggl\{\left\Vert N_{i}^{0}\right\Vert _{1},\left(\dfrac{1}{c\cos\theta}+\dfrac{\sin\theta}{\cos\theta}\right)\left\Vert N_{1}^{-}\right\Vert _{1},\left\Vert N_{1}^{--}\right\Vert _{1},\left(\dfrac{1}{c\sin\theta}+\dfrac{\cos\theta}{\sin\theta}\right)\left\Vert N_{2}^{+}\right\Vert _{1},\left\Vert N_{2}^{--}\right\Vert _{1},\\
\left(\dfrac{1}{c\sin\theta}+\dfrac{\cos\theta}{\sin\theta}\right)\left\Vert N_{3}^{-}\right\Vert _{1},\left\Vert N_{3}^{++}\right\Vert _{1},\left(\dfrac{1}{c\cos\theta}+\dfrac{\sin\theta}{\cos\theta}\right)\left\Vert N_{4}^{+}\right\Vert _{1},\left\Vert N_{4}^{++}\right\Vert _{1}\Biggr\}.\label{eq:koalla-1}
\end{multline}
 We similarly obtain 
\begin{multline}
\left\Vert \dfrac{\partial\mathcal{T}\left(M\right)}{\partial y}\right\Vert \leq\\
4cS\max\Biggl\{2T;2\left(b_{1}-a_{1}\right)\max\left\{ \dfrac{1}{c\cos\theta};\frac{1}{c\sin\theta}\right\} ;\\
\max\left\{ \dfrac{1}{c\cos\theta};\frac{1}{c\sin\theta}\right\} +2\left(b_{2}-a_{2}\right)\max\biggl\{\dfrac{1}{c\cos\theta}\left(\dfrac{1}{c\cos\theta}+\dfrac{\sin\theta}{\cos\theta}\right);\\
\dfrac{1}{c\sin\theta}\left(\dfrac{1}{c\sin\theta}+\dfrac{\cos\theta}{\sin\theta}\right)\biggr\}\Biggr\}\left(\mathscr{V}\left(M\right)\right)^{2}+\\
\max_{1\leq i\leq4}\Biggl\{\left\Vert N_{i}^{0}\right\Vert _{1},\left\Vert N_{1}^{-}\right\Vert _{1},\left(\dfrac{1}{c\sin\theta}+\dfrac{\cos\theta}{\sin\theta}\right)\left\Vert N_{1}^{--}\right\Vert _{1},\left\Vert N_{2}^{+}\right\Vert _{1},\left(\dfrac{1}{c\cos\theta}+\dfrac{\sin\theta}{\cos\theta}\right)\left\Vert N_{2}^{--}\right\Vert _{1},\\
\left\Vert N_{3}^{-}\right\Vert _{1},\left(\dfrac{1}{c\cos\theta}+\dfrac{\sin\theta}{\cos\theta}\right)\left\Vert N_{3}^{++}\right\Vert _{1},\left\Vert N_{4}^{+}\right\Vert _{1},\left(\dfrac{1}{c\sin\theta}+\dfrac{\cos\theta}{\sin\theta}\right)\left\Vert N_{4}^{++}\right\Vert _{1}\Biggr\}.\label{eq:koalla-1-1}
\end{multline}

\ref{eq:lospp}, \ref{eq:qmpql-1-2},\ref{eq:koalla-1} and \ref{eq:koalla-1-1}
imply that for $M\in\mathscr{E}^{4},$ 
\begin{multline}
\mathscr{V}\left(\mathcal{T}\left(M\right)\right)\leq4cS\max\Biggl\{\max\left\{ 1+2T\left(c\cos\theta+c\sin\theta\right);2T\right\} ;\\
\max\left\{ \dfrac{1}{c\cos\theta};\dfrac{1}{c\sin\theta}\right\} +2\max\left\{ \dfrac{1}{c\cos\theta};\dfrac{1}{c\sin\theta};\right.\\
\dfrac{1}{c\cos\theta}\left(\dfrac{1}{c\cos\theta}+\dfrac{\sin\theta}{\cos\theta}\right);\dfrac{1}{c\sin\theta}\left(\dfrac{1}{c\sin\theta}+\dfrac{\cos\theta}{\sin\theta}\right)\biggr\}\max\left\{ b_{1}-a_{1};b_{2}-a_{2}\right\} \Biggr\}\\
\left(\mathscr{V}\left(M\right)\right)^{2}+\\
{\displaystyle \max_{1\leq i\leq4}}\Biggl\{\max\left\{ 1;c\cos\theta+c\sin\theta\right\} \left\Vert N_{i}^{0}\right\Vert _{1};\\
\max\left\{ 1;\dfrac{1}{c\cos\theta}+\dfrac{\sin\theta}{\cos\theta}\right\} \left\Vert N_{1}^{-}\right\Vert _{1};\max\left\{ 1;\dfrac{1}{c\sin\theta}+\dfrac{\cos\theta}{\sin\theta}\right\} \left\Vert N_{1}^{--}\right\Vert _{1};\\
\max\left\{ 1;\dfrac{1}{c\sin\theta}+\dfrac{\cos\theta}{\sin\theta}\right\} \left\Vert N_{2}^{+}\right\Vert _{1};\max\left\{ 1;\dfrac{1}{c\cos\theta}+\dfrac{\sin\theta}{\cos\theta}\right\} \left\Vert N_{2}^{--}\right\Vert _{1};\\
\max\left\{ 1;\dfrac{1}{c\sin\theta}+\dfrac{\cos\theta}{\sin\theta}\right\} \left\Vert N_{3}^{-}\right\Vert _{1};\max\left\{ 1;\dfrac{1}{c\cos\theta}+\dfrac{\sin\theta}{\cos\theta}\right\} \left\Vert N_{3}^{++}\right\Vert _{1};\\
\max\left\{ 1;\dfrac{1}{c\cos\theta}+\dfrac{\sin\theta}{\cos\theta}\right\} \left\Vert N_{4}^{+}\right\Vert _{1};\max\left\{ 1;\dfrac{1}{c\sin\theta}+\dfrac{\cos\theta}{\sin\theta}\right\} \left\Vert N_{4}^{++}\right\Vert _{1}\Biggr\}\label{eq:olsoep-1}
\end{multline}
that is $\mathscr{V}\left(\mathcal{T}\left(M\right)\right)\leq p\left(\mathscr{V}\left(M\right)\right)^{2}+q$
where 
\begin{multline}
p\equiv4cs\max\Biggl\{\max\left\{ 1+2T\left(c\cos\theta+c\sin\theta\right);2T\right\} ;\\
\max\left\{ \dfrac{1}{c\cos\theta};\dfrac{1}{c\sin\theta}\right\} +2\max\left\{ \dfrac{1}{c\cos\theta};\dfrac{1}{c\sin\theta};\right.\\
\dfrac{1}{c\cos\theta}\left(\dfrac{1}{c\cos\theta}+\dfrac{\sin\theta}{\cos\theta}\right);\dfrac{1}{c\sin\theta}\left(\dfrac{1}{c\sin\theta}+\dfrac{\cos\theta}{\sin\theta}\right)\biggr\}\max\left\{ b_{1}-a_{1};b_{2}-a_{2}\right\} \Biggr\}\label{eq:ololoo}
\end{multline}

\begin{multline}
q\equiv{\displaystyle \max_{1\leq i\leq4}}\Biggl\{\max\left\{ 1;c\cos\theta+c\sin\theta\right\} \left\Vert N_{i}^{0}\right\Vert _{1};\\
\max\left\{ 1;\dfrac{1}{c\cos\theta}+\dfrac{\sin\theta}{\cos\theta}\right\} \left\Vert N_{1}^{-}\right\Vert _{1};\max\left\{ 1;\dfrac{1}{c\sin\theta}+\dfrac{\cos\theta}{\sin\theta}\right\} \left\Vert N_{1}^{--}\right\Vert _{1};\\
\max\left\{ 1;\dfrac{1}{c\sin\theta}+\dfrac{\cos\theta}{\sin\theta}\right\} \left\Vert N_{2}^{+}\right\Vert _{1};\max\left\{ 1;\dfrac{1}{c\cos\theta}+\dfrac{\sin\theta}{\cos\theta}\right\} \left\Vert N_{2}^{--}\right\Vert _{1};\\
\max\left\{ 1;\dfrac{1}{c\sin\theta}+\dfrac{\cos\theta}{\sin\theta}\right\} \left\Vert N_{3}^{-}\right\Vert _{1};\max\left\{ 1;\dfrac{1}{c\cos\theta}+\dfrac{\sin\theta}{\cos\theta}\right\} \left\Vert N_{3}^{++}\right\Vert _{1};\\
\max\left\{ 1;\dfrac{1}{c\cos\theta}+\dfrac{\sin\theta}{\cos\theta}\right\} \left\Vert N_{4}^{+}\right\Vert _{1};\max\left\{ 1;\dfrac{1}{c\sin\theta}+\dfrac{\cos\theta}{\sin\theta}\right\} \left\Vert N_{4}^{++}\right\Vert _{1}\Biggr\}\label{eq:kdipzp-1}
\end{multline}

If $M\in\mathscr{B}_{R}$ then $\mathscr{V}\left(\mathcal{T}\left(M\right)\right)\leq pR^{2}+q.$
So $\mathcal{T}\left(\mathscr{B}_{R}\right)\subset\mathscr{B}_{pR^{2}+q}.$
\end{proof}
\begin{lem}
\label{sllso}The operator $\mathcal{T}$ is compact on $\mathscr{B}_{R}$
forall $R>0.$ 
\end{lem}

\begin{proof}
$\mathcal{T}$ is continuous, $\mathcal{T}\left(\mathscr{B}_{R}\right)\subset\mathscr{B}_{pR^{2}+q}$
and $\mathscr{B}_{pR^{2}+q}$ is relatively compact in $\left(C\left(\begin{array}{r}
\mathscr{P}\end{array};\R^{4}\right),\left\Vert \cdot\right\Vert \right).$ Thus $\mathcal{T}$ is compact on $\mathscr{B}_{R}.$ 
\end{proof}
\begin{lem}
\label{prop::eolpa} If $pq\leq\dfrac{1}{4}$ and $\dfrac{1-\sqrt{1-4pq}}{2p}\leq R\leq\dfrac{1+\sqrt{1-4pq}}{2p}$
then $\mathcal{T}\left(\mathcal{\mathscr{B}}_{R}\right)\subset\mathcal{\mathcal{\mathscr{B}}}_{R}.$ 
\end{lem}

\begin{proof}
We have $pR^{2}-R+q\leq0.$ Hence $pR^{2}+q\leq R$. But $\mathcal{T}\left(\mathscr{B}_{R}\right)\subset\mathscr{B}_{pR^{2}+q}$
therefore $\mathcal{T}\left(\mathcal{\mathscr{B}}_{R}\right)\subset\mathcal{\mathcal{\mathscr{B}}}_{R}.$ 
\end{proof}
\begin{thm}
\label{sccssq}( Schauder \cite{17})

Let $\mathcal{M}$ be a non-empty convex subset of a normed space
$\mathscr{X}$ and $\mathcal{T}$ be a continuous compact mapping
from $\mathcal{M}$ into $\mathscr{\mathcal{M}}$ . Then $\mathcal{T}$
has a fixed point. 
\end{thm}

\subsection{Proof of the main theorem}
\begin{proof}
Let $\dfrac{1-\sqrt{1-4pq}}{2p}\leq R\leq\dfrac{1+\sqrt{1-4pq}}{2p}.$
From proposition (\ref{prop:The-sets-}), $\mathcal{\mathcal{\mathscr{B}}_{R}}$
is a non-empty convex subset of $\left(C\left(\begin{array}{r}
\mathscr{P}\end{array};\R^{4}\right),\left\Vert \cdot\right\Vert \right).$ From lemmas (\ref{lem:-is-continuous.}), (\ref{sllso}) and (\ref{prop::eolpa})
, $\mathcal{T}$ is continuous compact from $\mathcal{\mathcal{\mathscr{B}}_{R}}$
into $\mathcal{\mathcal{\mathcal{\mathscr{B}}_{R}}}.$ Thus according
to Schauder's theorem \ref{sccssq}, $\mathcal{T}$ has a fixed point
$N\in\mathcal{\mathcal{\mathscr{B}}_{R}}$ which is thus a solution
of problem $\Sigma$ and hence non-negative from theorem (\ref{thm::etaatyeyaiioa-1-1}).
$N\in\mathscr{B}_{R}$ also implies that $N\in\mathscr{E}^{4}$ and
$\mathscr{V}\left(N\right)\leq R\leq\dfrac{1+\sqrt{1-4pq}}{2p}.$
Thus $N\in C\left(\begin{array}{r}
\mathscr{P}\end{array};\R^{4}\right)$ and $\dfrac{\partial N_{i}}{\partial t},\dfrac{\partial N_{i}}{\partial x},\dfrac{\partial N_{i}}{\partial y}$
($i=1,2,3,4)$ are defined in $\mathring{\mathscr{P}},$ except possibly
on a finite number of planes , and continuous and bounded with $\left\Vert N\right\Vert ,\left\Vert \dfrac{\partial N}{\partial t}\right\Vert ,\left\Vert \dfrac{\partial N}{\partial x}\right\Vert ,\left\Vert \dfrac{\partial N}{\partial y}\right\Vert \leq\dfrac{1+\sqrt{1-4pq}}{2p}.$

Now, suppose that the problem $\Sigma$ has two solutions $M$ and
$N$ satisfying ${\displaystyle \left\Vert M\right\Vert ,\left\Vert N\right\Vert }\leq\dfrac{1+\sqrt{1-4pq}}{2p}.$
From relation (\ref{eq:lsoopa-2}) it holds
\begin{multline}
\left\Vert \mathcal{T}\left(M\right)-\mathcal{T}\left(N\right)\right\Vert \leq p'\cdot2\cdot\dfrac{1+\sqrt{1-4pq}}{2p}\left\Vert M-N\right\Vert \\
\leq\dfrac{p'}{p}\cdot\left(1+\sqrt{1-4pq}\right)\left\Vert M-N\right\Vert .\label{eq:lsoopa-1}
\end{multline}
As $M$ and $N$ are fixed points of $\mathcal{T},$ (\ref{eq:lsoopa-1})
rewrites
\begin{align}
\left\Vert M-N\right\Vert  & \leq\dfrac{p'}{p}\cdot\left(1+\sqrt{1-4pq}\right)\left\Vert M-N\right\Vert \label{eq:lsoopa-1-1}
\end{align}
i.e. 
\begin{align}
\left(1-\dfrac{p'}{p}\cdot\left(1+\sqrt{1-4pq}\right)\right)\left\Vert M-N\right\Vert  & \leq0.\label{eq:lsoopa-1-1-1}
\end{align}

From (\ref{eq:lsoopa-2}) and (\ref{eq:ololoo}) we have 
\begin{multline}
\dfrac{p'}{p}=\max\Biggl\{ T,\dfrac{b_{1}-a_{1}}{c\cos\theta},{\displaystyle \frac{b_{2}-a_{2}}{c\sin\theta}},\dfrac{b_{1}-a_{1}}{c\sin\theta},{\displaystyle \frac{b_{2}-a_{2}}{c\cos\theta}}\Biggl\}\times\\
\Biggl(\max\Biggl\{\max\left\{ 1+2T\left(c\cos\theta+c\sin\theta\right);2T\right\} ;\\
\max\left\{ \dfrac{1}{c\cos\theta};\dfrac{1}{c\sin\theta}\right\} +2\max\left\{ \dfrac{1}{c\cos\theta};\dfrac{1}{c\sin\theta};\right.\dfrac{1}{c\cos\theta}\left(\dfrac{1}{c\cos\theta}+\dfrac{\sin\theta}{\cos\theta}\right);\\
\dfrac{1}{c\sin\theta}\left(\dfrac{1}{c\sin\theta}+\dfrac{\cos\theta}{\sin\theta}\right)\biggr\}\max\left\{ b_{1}-a_{1};b_{2}-a_{2}\right\} \Biggr)^{-1}\\
\equiv\dfrac{\max\left\{ T,\alpha\right\} }{\max\left\{ \beta_{T},\beta\right\} }
\end{multline}
where 

\begin{multline}
\alpha=\max\left\{ \dfrac{b_{1}-a_{1}}{c\cos\theta},{\displaystyle \frac{b_{2}-a_{2}}{c\sin\theta}},\dfrac{b_{1}-a_{1}}{c\sin\theta},{\displaystyle \frac{b_{2}-a_{2}}{c\cos\theta}}\right\} ;\\
\beta_{T}=\max\left\{ 1+2T\left(c\cos\theta+c\sin\theta\right);2T\right\} 
\end{multline}
\begin{multline}
\beta=\max\left\{ \dfrac{1}{c\cos\theta};\dfrac{1}{c\sin\theta}\right\} +2\max\left\{ \dfrac{1}{c\cos\theta};\dfrac{1}{c\sin\theta};\right.\\
\dfrac{1}{c\cos\theta}\left(\dfrac{1}{c\cos\theta}+\dfrac{\sin\theta}{\cos\theta}\right);\dfrac{1}{c\sin\theta}\left(\dfrac{1}{c\sin\theta}+\dfrac{\cos\theta}{\sin\theta}\right)\biggr\}\max\left\{ b_{1}-a_{1};b_{2}-a_{2}\right\} .
\end{multline}
We have
\begin{equation}
\begin{cases}
\beta_{T}\geq2T\\
\beta>2\alpha
\end{cases}
\end{equation}
\begin{equation}
\max\left\{ \beta_{T},\beta\right\} \geq2\max\left\{ T,\alpha\right\} 
\end{equation}
and

\[
\dfrac{p'}{p}=\dfrac{\max\left\{ T,\alpha\right\} }{\max\left\{ \beta_{T},\beta\right\} }\leq\dfrac{1}{2}.
\]

As $1+\sqrt{1-4pq}<2,$ we have $\dfrac{p'}{p}\cdot\left(1+\sqrt{1-4pq}\right)<1$
i.e. $1-\dfrac{p'}{p}\cdot\left(1+\sqrt{1-4pq}\right)>0.$

From (\ref{eq:lsoopa-1-1-1}) it follows that $\left\Vert M-N\right\Vert \leq0$
therefore $M=N$ . Hence the uniqueness.
\end{proof}

\section{Conclusion}

In this paper we generalize the results obtained in \cite{key-1}
to the general four-velocity Broadwell model in the plane. Precisely,
we provide the proof of existence and uniqueness of classical positive
solutions for the initial-boundary value problem for this model in
a rectangular domain with Dirichlet boundary conditions. This work
represents a key step toward a broader mathematical understanding
of initial-boundary value problems for unsteady discrete kinetic models
in higher dimensions.

\appendix

\section*{Appendix}

\section{\label{sec:Solution-of-the}Solution of the system $\left(\Sigma_{M}\right)$}

\begin{multline}
\mathcal{T}_{1}\left(M\right)\left(t,x,y\right)=\mathcal{T}_{1}^{A}\left(M\right)\left(t,x,y\right)\cdot\mathbb{I}_{\begin{cases}
x-ct\cos\theta\geq a_{1}\\
y-ct\sin\theta\geq a_{2}
\end{cases}}\left(t,x,y\right)+\\
\mathcal{T}_{1}^{B}\left(M\right)\left(t,x,y\right)\cdot\mathbb{I}_{\begin{cases}
x-ct\cos\theta\leq a_{1}\\
x\sin\theta-y\cos\theta\leq a_{1}\sin\theta-a_{2}\cos\theta
\end{cases}}\left(t,x,y\right)+\\
\mathcal{T}_{1}^{C}\left(M\right)\left(t,x,y\right)\cdot\mathbb{I}_{\begin{cases}
y-ct\sin\theta\leq a_{2}\\
x\sin\theta-y\cos\theta\geq a_{1}\sin\theta-a_{2}\cos\theta
\end{cases}}\left(t,x,y\right)\label{eq:oloooso}
\end{multline}
where 
\begin{multline}
\mathcal{T}_{1}^{A}\left(M\right)\left(t,x,y\right)={\displaystyle \int_{0}^{t}}Q\left(M\right)\left(\right.s,x+c\left(s-t\right)\cos\theta,\\
y+c\left(s-t\right)\sin\theta\left.\right)ds+N_{1}^{0}\left(x-ct\cos\theta,y-ct\sin\theta\right)\label{eq:ikiiz}
\end{multline}
\begin{multline}
\mathcal{T}_{1}^{B}\left(M\right)\left(t,x,y\right)={\textstyle {\displaystyle \int_{0}^{\frac{1}{c\cos\theta}x-\frac{a_{1}}{c\cos\theta}}}}{\textstyle Q\left(M\right)}\left(\right.s+t-{\textstyle \frac{1}{c\cos\theta}}x+{\textstyle \frac{a_{1}}{c\cos\theta}},\\
sc\cos\theta+a_{1},sc\sin\theta-{\textstyle \frac{\sin\theta}{\cos\theta}}x+y+{\textstyle \frac{\sin\theta}{\cos\theta}}a_{1}\left.\right)ds+\\
{\textstyle N_{1}^{-}\left(t-\frac{1}{c\cos\theta}x+\frac{a_{1}}{c\cos\theta},-\frac{\sin\theta}{\cos\theta}x+y+\frac{\sin\theta}{\cos\theta}a_{1}\right)}
\end{multline}
\begin{multline}
\mathcal{T}_{1}^{C}\left(M\right)\left(t,x,y\right)={\textstyle {\displaystyle \int_{0}^{\frac{1}{c\sin\theta}y-\frac{a_{2}}{c\sin\theta}}}}{\textstyle Q\left(M\right)}\left(\right.s+t-{\textstyle \frac{1}{c\sin\theta}}y+{\textstyle \frac{a_{2}}{c\sin\theta}},\\
sc\cos\theta+x-{\textstyle \frac{\cos\theta}{\sin\theta}}y+{\textstyle \frac{\cos\theta}{\sin\theta}}a_{2},sc\sin\theta+a_{2}\left.\right)ds+\\
{\textstyle N_{1}^{--}\left(t-\frac{1}{c\sin\theta}y+\frac{a_{2}}{c\sin\theta},x-\frac{\cos\theta}{\sin\theta}y+\frac{\cos\theta}{\sin\theta}a_{2}\right)}\label{eq:looaoolzo}
\end{multline}
\begin{multline}
\mathcal{T}_{2}\left(M\right)\left(t,x,y\right)=\mathcal{T}_{2}^{A}\left(M\right)\left(t,x,y\right)\cdot\mathbb{I}_{\begin{cases}
x-ct\cos\theta\geq a_{1}\\
y-ct\sin\theta\geq a_{2}
\end{cases}}\left(t,x,y\right)+\\
\mathcal{T}_{2}^{B}\left(M\right)\left(t,x,y\right)\cdot\mathbb{I}_{\begin{cases}
x-ct\cos\theta\leq a_{1}\\
x\sin\theta-y\cos\theta\leq a_{1}\sin\theta-a_{2}\cos\theta
\end{cases}}\left(t,x,y\right)+\\
\mathcal{T}_{2}^{C}\left(M\right)\left(t,x,y\right)\cdot\mathbb{I}_{\begin{cases}
y-ct\sin\theta\leq a_{2}\\
x\sin\theta-y\cos\theta\geq a_{1}\sin\theta-a_{2}\cos\theta
\end{cases}}\left(t,x,y\right)\label{eq:ppmms-2-1}
\end{multline}
where

\begin{multline}
\mathcal{T}_{2}^{A}\left(M\right)\left(t,x,y\right)={\textstyle {\displaystyle \int_{0}^{t}}}-Q\left(M\right)\left(\right.s,x-c\left(s-t\right)\sin\theta,\\
y+c\left(s-t\right)\cos\theta\left.\right)ds+{\textstyle N_{2}^{0}\left(x+ct\sin\theta,y-ct\cos\theta\right)}\label{eq:zoppz}
\end{multline}
\begin{multline}
\mathcal{T}_{2}^{B}\left(M\right)\left(t,x,y\right)={\textstyle {\displaystyle \int_{0}^{-\frac{1}{c\sin\theta}x+\frac{b_{1}}{c\sin\theta}}}}-{\textstyle Q\left(M\right)}\Bigl(s+t+{\displaystyle \frac{1}{c\sin\theta}}x-\frac{b_{1}}{c\sin\theta},\\
-sc\sin\theta+b_{1},sc\cos\theta+\frac{\cos\theta}{\sin\theta}x+y-\frac{\cos\theta}{\sin\theta}b_{1}\Bigr)ds+\\
{\textstyle N_{2}^{+}\left(t+\frac{1}{c\sin\theta}x-\frac{b_{1}}{c\sin\theta},\frac{\cos\theta}{\sin\theta}x+y-\frac{\cos\theta}{\sin\theta}b_{1}\right)}\cdot\label{eq:opps}
\end{multline}
\begin{multline}
\mathcal{T}_{2}^{C}\left(M\right)\left(t,x,y\right)={\textstyle {\displaystyle \int_{0}^{\frac{1}{c\cos\theta}y-\frac{a_{2}}{c\cos\theta}}}}{\textstyle -Q\left(M\right)}\Bigl(s+t-\frac{1}{c\cos\theta}y+\frac{a_{2}}{c\cos\theta},\\
-sc\sin\theta+x+\frac{\sin\theta}{\cos\theta}y-\frac{\sin\theta}{\cos\theta}a_{2},sc\cos\theta+a_{2}\Bigr)ds+\\
{\textstyle N_{2}^{--}\left(t-\frac{1}{c\cos\theta}y+\frac{a_{2}}{c\cos\theta},x+\frac{\sin\theta}{\cos\theta}y-\frac{\sin\theta}{\cos\theta}a_{2}\right)}\cdot\label{eq:opos}
\end{multline}

\begin{multline}
\mathcal{T}_{3}\left(M\right)\left(t,x,y\right)=\mathcal{T}_{3}^{A}\left(M\right)\left(t,x,y\right)\cdot\mathbb{I}_{\begin{cases}
x-ct\sin\theta\geq a_{1}\\
y+ct\cos\theta\leq b_{2}
\end{cases}}\left(t,x,y\right)+\\
\mathcal{T}_{3}^{B}\left(M\right)\left(t,x,y\right)\cdot\mathbb{I}_{\begin{cases}
x-ct\sin\theta\leq a_{1}\\
x\cos\theta+y\sin\theta\leq a_{1}\cos\theta+b_{2}\sin\theta
\end{cases}}\left(t,x,y\right)+\\
\mathcal{T}_{3}^{C}\left(M\right)\left(t,x,y\right)\cdot\mathbb{I}_{\begin{cases}
y+ct\cos\theta\geq b_{2}\\
x\cos\theta+y\sin\theta\geq a_{1}\cos\theta+b_{2}\sin\theta
\end{cases}}\left(t,x,y\right)\label{aoalal-1-1-2}
\end{multline}
where 

\begin{multline}
\mathcal{T}_{3}^{A}\left(M\right)\left(t,x,y\right)={\displaystyle \int_{0}^{t}}-Q\left(M\right)\left(s,x+c\left(s-t\right)\sin\theta,y-c\left(s-t\right)\cos\theta\right)ds+\\
N_{3}^{0}\left(x-ct\sin\theta,y+ct\cos\theta\right)\label{eq:lqooqsd}
\end{multline}
\begin{multline*}
\mathcal{T}_{3}^{B}\left(M\right)\left(t,x,y\right)={\textstyle {\displaystyle \int_{0}^{\frac{1}{c\sin\theta}x-\frac{a_{1}}{c\sin\theta}}}}{\textstyle -Q\left(M\right)}\Bigl(s+t-\frac{1}{c\sin\theta}x+\frac{a_{1}}{c\sin\theta},\\
sc\sin\theta+a_{1},-sc\cos\theta+\frac{\cos\theta}{\sin\theta}x+y-\frac{\cos\theta}{\sin\theta}a_{1}\Bigr)ds+\\
{\textstyle N_{3}^{-}\left(t-\frac{1}{c\sin\theta}x+\frac{a_{1}}{c\sin\theta},\frac{\cos\theta}{\sin\theta}x+y-\frac{\cos\theta}{\sin\theta}a_{1}\right)}
\end{multline*}
\begin{multline}
\mathcal{T}_{3}^{C}\left(M\right)\left(t,x,y\right)={\textstyle {\displaystyle \int_{0}^{-\frac{1}{c\cos\theta}y+\frac{b_{2}}{c\cos\theta}}}}{\textstyle -Q\left(M\right)}\Bigl(s+t+\frac{1}{c\cos\theta}y-\frac{b_{2}}{c\cos\theta},\\
sc\sin\theta+x+\frac{\sin\theta}{\cos\theta}y-\frac{\sin\theta}{\cos\theta}b_{2},-sc\cos\theta+b_{2}\Bigr)ds+\\
{\textstyle N_{3}^{++}\left(t+\frac{1}{c\cos\theta}y-\frac{b_{2}}{c\cos\theta},x+\frac{\sin\theta}{\cos\theta}y-\frac{\sin\theta}{\cos\theta}b_{2}\right)}\label{eq:loiuij}
\end{multline}

\begin{multline}
\mathcal{T}_{4}\left(M\right)\left(t,x,y\right)=\mathcal{T}_{4}^{A}\left(M\right)\left(t,x,y\right)\cdot\mathbb{I}_{\begin{cases}
x+ct\cos\theta\leq b_{1}\\
y+ct\sin\theta\leq b_{2}
\end{cases}}\left(t,x,y\right)+\\
\mathcal{T}_{4}^{B}\left(M\right)\left(t,x,y\right)\cdot\mathbb{I}_{\begin{cases}
x+ct\cos\theta\geq b_{1}\\
x\sin\theta-y\cos\theta\geq b_{1}\sin\theta-b_{2}\cos\theta
\end{cases}}\left(t,x,y\right)+\\
\mathcal{T}_{4}^{C}\left(M\right)\left(t,x,y\right)\cdot\mathbb{I}_{\begin{cases}
y+ct\sin\theta\geq b_{2}\\
x\sin\theta-y\cos\theta\leq b_{1}\sin\theta-b_{2}\cos\theta
\end{cases}}\left(t,x,y\right)\label{aoalal-1-1-1-2}
\end{multline}
where 
\begin{multline}
\mathcal{T}_{4}^{A}\left(M\right)\left(t,x,y\right)={\displaystyle \int_{0}^{t}}Q\left(M\right)\left(s,x-c\left(s-t\right)\cos\theta,y-c\left(s-t\right)\sin\theta\right)ds+\\
N_{4}^{0}\left(x+ct\cos\theta,y+ct\sin\theta\right)\label{eq:olole-1-2-1-1}
\end{multline}
\begin{multline}
\mathcal{T}_{4}^{B}\left(M\right)\left(t,x,y\right)=\Biggl({\textstyle {\displaystyle \int_{0}^{-\frac{1}{c\cos\theta}x+\frac{b_{1}}{c\cos\theta}}}}{\textstyle Q\left(M\right)}\Bigl(s+t+\frac{1}{c\cos\theta}x-\frac{b_{1}}{c\cos\theta},\\
-sc\cos\theta+b_{1},-sc\sin\theta-\frac{\sin\theta}{\cos\theta}x+y+\frac{\sin\theta}{\cos\theta}b_{1}\Bigr)ds+\\
{\textstyle N_{4}^{+}\left(t+\frac{1}{c\cos\theta}x-\frac{b_{1}}{c\cos\theta},-\frac{\sin\theta}{\cos\theta}x+y+\frac{\sin\theta}{\cos\theta}b_{1}\right)}\label{eq:olole-1-1-1-2-1}
\end{multline}
\begin{multline}
\mathcal{T}_{4}^{C}\left(M\right)\left(t,x,y\right)={\textstyle {\displaystyle \int_{0}^{-\frac{1}{c\sin\theta}y+\frac{b_{2}}{c\sin\theta}}}}{\textstyle Q\left(M\right)}\Bigl(s+t+\frac{1}{c\sin\theta}y-\frac{b_{2}}{c\sin\theta},\\
-sc\cos\theta+x-\frac{\cos\theta}{\sin\theta}y+\frac{\cos\theta}{\sin\theta}b_{2},-sc\sin\theta+b_{2}\Bigr)ds+\\
{\textstyle N_{4}^{++}\left(t+\frac{1}{c\sin\theta}y-\frac{b_{2}}{c\sin\theta},x-\frac{\cos\theta}{\sin\theta}y+\frac{\cos\theta}{\sin\theta}b_{2}\right)}\cdot\label{eq:olole-1-1-1-1-1-2-1}
\end{multline}

\section{\label{sec:Derivatives}Derivatives $\dfrac{\partial\mathcal{T}_{i}\left(M\right)}{\partial t},\dfrac{\partial\mathcal{T}_{i}\left(M\right)}{\partial x},\dfrac{\partial\mathcal{T}_{i}\left(M\right)}{\partial y},\left(i=1,2,3,4\right)$}

$\dfrac{\partial\mathcal{T}_{1}\left(M\right)}{\partial t},\dfrac{\partial\mathcal{T}_{1}\left(M\right)}{\partial x},\dfrac{\partial\mathcal{T}_{1}\left(M\right)}{\partial y}$
are defined in $\mathring{\mathscr{P}}$ except perhaps on the planes
given respectively by the equations $x-ct\cos\theta=a_{1};y-ct\sin\theta=a_{2};x\sin\theta-y\cos\theta=a_{1}\sin\theta-a_{2}\cos\theta$
and
\begin{multline}
\left(\dfrac{\partial}{\partial t},\dfrac{\partial}{\partial x},\dfrac{\partial}{\partial y}\right)\mathcal{T}_{1}\left(M\right)\left(t,x,y\right)=\\
\left(\dfrac{\partial}{\partial t},\dfrac{\partial}{\partial x},\dfrac{\partial}{\partial y}\right)\mathcal{T}_{1}^{A}\left(M\right)\left(t,x,y\right)\cdot\mathbb{I}_{\begin{cases}
x-ct\cos\theta>a_{1}\\
y-ct\sin\theta>a_{2}
\end{cases}}\left(t,x,y\right)+\\
\left(\dfrac{\partial}{\partial t},\dfrac{\partial}{\partial x},\dfrac{\partial}{\partial y}\right)\mathcal{T}_{1}^{B}\left(M\right)\left(t,x,y\right)\cdot\\
\mathbb{I}_{\begin{cases}
x-ct\cos\theta<a_{1}\\
x\sin\theta-y\cos\theta<a_{1}\sin\theta-a_{2}\cos\theta
\end{cases}}\left(t,x,y\right)+\\
\left(\dfrac{\partial}{\partial t},\dfrac{\partial}{\partial x},\dfrac{\partial}{\partial y}\right)\mathcal{T}_{1}^{C}\left(M\right)\left(t,x,y\right)\cdot\mathbb{I}_{\begin{cases}
y-ct\sin\theta<a_{2}\\
x\sin\theta-y\cos\theta>a_{1}\sin\theta-a_{2}\cos\theta
\end{cases}}\left(t,x,y\right)\label{eq:oloooso-1}
\end{multline}
where 
\begin{multline}
\dfrac{\partial\mathcal{T}_{1}^{A}\left(M\right)}{\partial t}\left(t,x,y\right)=Q\left(M\right)\left(t,x,y\right)+{\displaystyle \int_{0}^{t}}\left[\right.-c\cos\theta\dfrac{\partial Q\left(M\right)}{\partial x}\\
\left(\right.s,x+c\left(s-t\right)\cos\theta,y+c\left(s-t\right)\sin\theta\left.\right)-c\sin\theta\dfrac{\partial Q\left(M\right)}{\partial y}\\
\left(\right.s,x+c\left(s-t\right)\cos\theta,y+c\left(s-t\right)\sin\theta\left.\right)\left.\right]ds\\
-c\cos\theta\dfrac{\partial N_{1}^{0}}{\partial x}\left(x-ct\cos\theta,y-ct\sin\theta\right)-c\sin\theta\dfrac{\partial N_{1}^{0}}{\partial y}\left(x-ct\cos\theta,y-ct\sin\theta\right)
\end{multline}
\begin{multline}
\dfrac{\partial\mathcal{T}_{1}^{B}\left(M\right)}{\partial t}\left(t,x,y\right)={\textstyle {\displaystyle \int_{0}^{\frac{1}{c\cos\theta}x-\frac{a_{1}}{c\cos\theta}}}}{\textstyle \dfrac{\partial Q\left(M\right)}{\partial t}}\left(\right.s+t-{\textstyle \frac{1}{c\cos\theta}}x+{\textstyle \frac{a_{1}}{c\cos\theta}},\\
sc\cos\theta+a_{1},sc\sin\theta-{\textstyle \frac{\sin\theta}{\cos\theta}}x+y+{\textstyle \frac{\sin\theta}{\cos\theta}}a_{1}\left.\right)ds+\\
{\textstyle \dfrac{\partial N_{1}^{-}}{\partial t}\left(t-\frac{1}{c\cos\theta}x+\frac{a_{1}}{c\cos\theta},-\frac{\sin\theta}{\cos\theta}x+y+\frac{\sin\theta}{\cos\theta}a_{1}\right)}
\end{multline}
\begin{multline}
\dfrac{\partial\mathcal{T}_{1}^{C}\left(M\right)}{\partial t}\left(t,x,y\right)={\textstyle {\displaystyle \int_{0}^{\frac{1}{c\sin\theta}y-\frac{a_{2}}{c\sin\theta}}}}{\textstyle \dfrac{\partial Q\left(M\right)}{\partial t}}\left(\right.s+t-{\textstyle \frac{1}{c\sin\theta}}y+{\textstyle \frac{a_{2}}{c\sin\theta}},\\
sc\cos\theta+x-{\textstyle \frac{\cos\theta}{\sin\theta}}y+{\textstyle \frac{\cos\theta}{\sin\theta}}a_{2},sc\sin\theta+a_{2}\left.\right)ds+\\
{\textstyle \dfrac{\partial N_{1}^{--}}{\partial t}\left(t-\frac{1}{c\sin\theta}y+\frac{a_{2}}{c\sin\theta},x-\frac{\cos\theta}{\sin\theta}y+\frac{\cos\theta}{\sin\theta}a_{2}\right)}\label{eq:looaoolzo-1}
\end{multline}
\begin{multline}
\dfrac{\partial\mathcal{T}_{1}^{A}\left(M\right)}{\partial x}\left(t,x,y\right)=\\
{\displaystyle \int_{0}^{t}}\left[\right.\dfrac{\partial Q\left(M\right)}{\partial x}\left(\right.s,x+c\left(s-t\right)\cos\theta,y+c\left(s-t\right)\sin\theta\left.\right)ds+\\
\dfrac{\partial N_{1}^{0}}{\partial x}\left(x-ct\cos\theta,y-ct\sin\theta\right)
\end{multline}
\begin{multline}
\dfrac{\partial\mathcal{T}_{1}^{B}\left(M\right)}{\partial x}\left(t,x,y\right)={\textstyle \frac{1}{c\cos\theta}}Q\left(M\right)\left(t,x,y\right)+\\
{\textstyle {\displaystyle \int_{0}^{\frac{1}{c\cos\theta}x-\frac{a_{1}}{c\cos\theta}}}}\left(\right.-{\textstyle \frac{1}{c\cos\theta}}{\textstyle \dfrac{\partial Q\left(M\right)}{\partial t}}\left(\right.s+t-{\textstyle \frac{1}{c\cos\theta}}x+{\textstyle \frac{a_{1}}{c\cos\theta}},\\
sc\cos\theta+a_{1},sc\sin\theta-{\textstyle \frac{\sin\theta}{\cos\theta}}x+y+{\textstyle \frac{\sin\theta}{\cos\theta}}a_{1}\left.\right)\\
-{\textstyle \frac{\sin\theta}{\cos\theta}}{\textstyle \dfrac{\partial Q\left(M\right)}{\partial y}}\left(\right.s+t-{\textstyle \frac{1}{c\cos\theta}}x+{\textstyle \frac{a_{1}}{c\cos\theta}},sc\cos\theta+a_{1},\\
sc\sin\theta-{\textstyle \frac{\sin\theta}{\cos\theta}}x+y+{\textstyle \frac{\sin\theta}{\cos\theta}}a_{1}\left.\right)\left.\right)ds+\\
-\frac{1}{c\cos\theta}{\textstyle \dfrac{\partial N_{1}^{-}}{\partial t}\left(t-\frac{1}{c\cos\theta}x+\frac{a_{1}}{c\cos\theta},-\frac{\sin\theta}{\cos\theta}x+y+\frac{\sin\theta}{\cos\theta}a_{1}\right)}\\
-{\textstyle {\displaystyle \frac{\sin\theta}{\cos\theta}}}\dfrac{\partial N_{1}^{-}}{\partial t}\left(t-\frac{1}{c\cos\theta}x+\frac{a_{1}}{c\cos\theta},-\frac{\sin\theta}{\cos\theta}x+y+\frac{\sin\theta}{\cos\theta}a_{1}\right)
\end{multline}
\begin{multline}
\dfrac{\partial\mathcal{T}_{1}^{C}\left(M\right)}{\partial x}\left(t,x,y\right)={\textstyle {\displaystyle \int_{0}^{\frac{1}{c\sin\theta}y-\frac{a_{2}}{c\sin\theta}}}}{\textstyle \dfrac{\partial Q\left(M\right)}{\partial y}}\left(\right.s+t-{\textstyle \frac{1}{c\sin\theta}}y+{\textstyle \frac{a_{2}}{c\sin\theta}},\\
sc\cos\theta+x-{\textstyle \frac{\cos\theta}{\sin\theta}}y+{\textstyle \frac{\cos\theta}{\sin\theta}}a_{2},sc\sin\theta+a_{2}\left.\right)ds+\\
{\textstyle \dfrac{\partial N_{1}^{--}}{\partial x}\left(t-\frac{1}{c\sin\theta}y+\frac{a_{2}}{c\sin\theta},x-\frac{\cos\theta}{\sin\theta}y+\frac{\cos\theta}{\sin\theta}a_{2}\right)}\label{eq:looaoolzo-1-1}
\end{multline}
\begin{multline}
\dfrac{\partial\mathcal{T}_{1}^{A}\left(M\right)}{\partial y}\left(t,x,y\right)=\\
{\displaystyle \int_{0}^{t}}\left[\right.\dfrac{\partial Q\left(M\right)}{\partial y}\left(\right.s,x+c\left(s-t\right)\cos\theta,y+c\left(s-t\right)\sin\theta\left.\right)ds+\\
\dfrac{\partial N_{1}^{0}}{\partial y}\left(x-ct\cos\theta,y-ct\sin\theta\right)
\end{multline}
\begin{multline}
\dfrac{\partial\mathcal{T}_{1}^{B}\left(M\right)}{\partial y}\left(t,x,y\right)=\\
{\textstyle {\displaystyle \int_{0}^{\frac{1}{c\cos\theta}x-\frac{a_{1}}{c\cos\theta}}}}{\textstyle \dfrac{\partial Q\left(M\right)}{\partial y}}\left(\right.s+t-{\textstyle \frac{1}{c\cos\theta}}x+{\textstyle \frac{a_{1}}{c\cos\theta}},\\
sc\cos\theta+a_{1},sc\sin\theta-{\textstyle \frac{\sin\theta}{\cos\theta}}x+y+{\textstyle \frac{\sin\theta}{\cos\theta}}a_{1}\left.\right)ds+\\
{\textstyle \dfrac{\partial N_{1}^{-}}{\partial y}\left(t-\frac{1}{c\cos\theta}x+\frac{a_{1}}{c\cos\theta},-\frac{\sin\theta}{\cos\theta}x+y+\frac{\sin\theta}{\cos\theta}a_{1}\right)}
\end{multline}
\begin{multline}
\dfrac{\partial\mathcal{T}_{1}^{C}\left(M\right)}{\partial y}\left(t,x,y\right)=\\
{\textstyle \frac{1}{c\sin\theta}}Q\left(M\right)\left(t,x,y\right)+{\textstyle {\displaystyle \int_{0}^{\frac{1}{c\sin\theta}y-\frac{a_{1}}{c\sin\theta}}}}\left(\right.-{\textstyle \frac{1}{c\sin\theta}}{\textstyle \dfrac{\partial Q\left(M\right)}{\partial t}}\left(\right.s+t-{\textstyle \frac{1}{c\sin\theta}}y+{\textstyle \frac{a_{2}}{c\sin\theta}},\\
sc\cos\theta+x-{\textstyle \frac{\cos\theta}{\sin\theta}}y+{\textstyle \frac{\cos\theta}{\sin\theta}}a_{2},sc\sin\theta+a_{2}\left.\right)\\
-{\textstyle \frac{\cos\theta}{\sin\theta}}{\textstyle \dfrac{\partial Q\left(M\right)}{\partial y}}\left(\right.s+t-{\textstyle \frac{1}{c\sin\theta}}y+{\textstyle \frac{a_{2}}{c\sin\theta}},sc\cos\theta+x-{\textstyle \frac{\cos\theta}{\sin\theta}}y+{\textstyle \frac{\cos\theta}{\sin\theta}}a_{2},sc\sin\theta+a_{2}\left.\right)\left.\right)ds+\\
-\frac{1}{c\sin\theta}{\textstyle \dfrac{\partial N_{1}^{--}}{\partial t}\left(t-\frac{1}{c\sin\theta}y+\frac{a_{2}}{c\sin\theta},x-\frac{\cos\theta}{\sin\theta}y+\frac{\cos\theta}{\sin\theta}a_{2}\right)}\\
-{\textstyle {\displaystyle \frac{\cos\theta}{\sin\theta}}}\dfrac{\partial N_{1}^{--}}{\partial x}\left(t-\frac{1}{c\sin\theta}y+\frac{a_{2}}{c\sin\theta},x-\frac{\cos\theta}{\sin\theta}y+\frac{\cos\theta}{\sin\theta}a_{2}\right)
\end{multline}
{*}{*}{*}{*} $\dfrac{\partial\mathcal{T}_{2}\left(M\right)}{\partial t},\dfrac{\partial\mathcal{T}_{2}\left(M\right)}{\partial x},\dfrac{\partial\mathcal{T}_{2}\left(M\right)}{\partial y}$
are defined in $\mathring{\mathscr{P}}$ except perhaps on the planes
given respectively by the equations $x+ct\sin\theta=b_{1};y-ct\cos\theta=a_{2};x\cos\theta+y\sin\theta=b_{1}\cos\theta+a_{2}\sin\theta$
and
\begin{multline}
\left(\dfrac{\partial}{\partial t},\dfrac{\partial}{\partial x},\dfrac{\partial}{\partial y}\right)\mathcal{T}_{2}\left(M\right)\left(t,x,y\right)=\\
\left(\dfrac{\partial}{\partial t},\dfrac{\partial}{\partial x},\dfrac{\partial}{\partial y}\right)\mathcal{T}_{2}^{A}\left(M\right)\left(t,x,y\right)\cdot\mathbb{I}_{\begin{cases}
x+ct\sin\theta<b_{1}\\
y-ct\cos\theta>a_{2}
\end{cases}}\left(t,x,y\right)+\\
\left(\dfrac{\partial}{\partial t},\dfrac{\partial}{\partial x},\dfrac{\partial}{\partial y}\right)\mathcal{T}_{2}^{B}\left(M\right)\left(t,x,y\right)\cdot\mathbb{I}_{\begin{cases}
x+ct\sin\theta>b_{1}\\
x\cos\theta+y\sin\theta>b_{1}\cos\theta+a_{2}\sin\theta
\end{cases}}\left(t,x,y\right)+\\
\left(\dfrac{\partial}{\partial t},\dfrac{\partial}{\partial x},\dfrac{\partial}{\partial y}\right)\mathcal{T}_{2}^{C}\left(M\right)\left(t,x,y\right)\cdot\mathbb{I}_{\begin{cases}
y-ct\cos\theta<a_{2}\\
x\cos\theta+y\sin\theta<b_{1}\cos\theta+a_{2}\sin\theta
\end{cases}}\left(t,x,y\right)\label{eq:ppmms-2-1-1-1-1}
\end{multline}
where 
\begin{multline}
\dfrac{\partial\mathcal{T}_{2}^{A}\left(M\right)}{\partial t}\left(t,x,y\right)=-Q\left(M\right)\left(t,x,y\right)+{\displaystyle \int_{0}^{t}}\left[\right.-c\sin\theta\dfrac{\partial Q\left(M\right)}{\partial x}\\
\left(\right.s,x-c\left(s-t\right)\sin\theta,y+c\left(s-t\right)\cos\theta\left.\right)+c\cos\theta\dfrac{\partial Q\left(M\right)}{\partial y}\\
\left(\right.s,x-c\left(s-t\right)\sin\theta,y+c\left(s-t\right)\cos\theta\left.\right)\left.\right]ds\\
+c\sin\theta\dfrac{\partial N_{2}^{0}}{\partial x}\left(x+ct\sin\theta,y-ct\cos\theta\right)-c\cos\theta\dfrac{\partial N_{2}^{0}}{\partial y}\left(x+ct\sin\theta,y-ct\cos\theta\right)
\end{multline}
\begin{multline}
\dfrac{\partial\mathcal{T}_{2}^{B}\left(M\right)}{\partial t}\left(t,x,y\right)={\textstyle {\displaystyle \int_{0}^{-\frac{1}{c\sin\theta}x+\frac{b_{1}}{c\sin\theta}}}}{\textstyle -\dfrac{\partial Q\left(M\right)}{\partial t}}\left(\right.s+t+{\displaystyle \frac{1}{c\sin\theta}}x-\frac{b_{1}}{c\sin\theta},\\
-sc\sin\theta+b_{1},sc\cos\theta+\frac{\cos\theta}{\sin\theta}x+y-\frac{\cos\theta}{\sin\theta}b_{1}\left.\right)ds+\\
{\textstyle \dfrac{\partial N_{2}^{+}}{\partial t}\left(t+\frac{1}{c\sin\theta}x-\frac{b_{1}}{c\sin\theta},\frac{\cos\theta}{\sin\theta}x+y-\frac{\cos\theta}{\sin\theta}b_{1}\right)}
\end{multline}
\begin{multline}
\dfrac{\partial\mathcal{T}_{2}^{C}\left(M\right)}{\partial t}\left(t,x,y\right)={\textstyle {\displaystyle \int_{0}^{\frac{1}{c\cos\theta}y-\frac{a_{2}}{c\cos\theta}}}}-{\textstyle \dfrac{\partial Q\left(M\right)}{\partial t}}\left(\right.s+t-\frac{1}{c\cos\theta}y+\frac{a_{2}}{c\cos\theta},\\
-sc\sin\theta+x+\frac{\sin\theta}{\cos\theta}y-\frac{\sin\theta}{\cos\theta}a_{2},sc\cos\theta+a_{2}\left.\right)ds+\\
{\textstyle \dfrac{\partial N_{2}^{--}}{\partial t}\left(t-\frac{1}{c\cos\theta}y+\frac{a_{2}}{c\cos\theta},x+\frac{\sin\theta}{\cos\theta}y-\frac{\sin\theta}{\cos\theta}a_{2}\right)}\label{eq:looaoolzo-1-2}
\end{multline}
\begin{multline}
\dfrac{\partial\mathcal{T}_{2}^{A}\left(M\right)}{\partial x}\left(t,x,y\right)=\\
{\displaystyle \int_{0}^{t}}\left[\right.-\dfrac{\partial Q\left(M\right)}{\partial x}\left(\right.s,x-c\left(s-t\right)\sin\theta,y+c\left(s-t\right)\cos\theta\left.\right)ds+\\
\dfrac{\partial N_{2}^{0}}{\partial x}\left(x+ct\sin\theta,y-ct\cos\theta\right)
\end{multline}
\begin{multline}
\dfrac{\partial\mathcal{T}_{2}^{B}\left(M\right)}{\partial x}\left(t,x,y\right)={\textstyle \frac{1}{c\sin\theta}}Q\left(M\right)\left(t,x,y\right)+\\
{\textstyle {\displaystyle \int_{0}^{-\frac{1}{c\sin\theta}x+\frac{b_{1}}{c\sin\theta}}}}\left(\right.-{\textstyle \frac{1}{c\sin\theta}}{\textstyle \dfrac{\partial Q\left(M\right)}{\partial t}}\left(\right.s+t+{\displaystyle \frac{1}{c\sin\theta}}x-\frac{b_{1}}{c\sin\theta},\\
-sc\sin\theta+b_{1},sc\cos\theta+\frac{\cos\theta}{\sin\theta}x+y-\frac{\cos\theta}{\sin\theta}b_{1}\left.\right)\\
-{\textstyle \frac{\cos\theta}{\sin\theta}}{\textstyle \dfrac{\partial Q\left(M\right)}{\partial y}}\left(\right.s+t+{\displaystyle \frac{1}{c\sin\theta}}x-\frac{b_{1}}{c\sin\theta},-sc\sin\theta+b_{1},\\
sc\cos\theta+\frac{\cos\theta}{\sin\theta}x+y-\frac{\cos\theta}{\sin\theta}b_{1}\left.\right)\left.\right)ds+\\
+\frac{1}{c\sin\theta}{\textstyle \dfrac{\partial N_{2}^{+}}{\partial t}\left(t+\frac{1}{c\sin\theta}x-\frac{b_{1}}{c\sin\theta},\frac{\cos\theta}{\sin\theta}x+y-\frac{\cos\theta}{\sin\theta}b_{1}\right)}\\
+{\textstyle {\displaystyle \frac{\cos\theta}{\sin\theta}}}\dfrac{\partial N_{2}^{+}}{\partial t}\left(t+\frac{1}{c\sin\theta}x-\frac{b_{1}}{c\sin\theta},\frac{\cos\theta}{\sin\theta}x+y-\frac{\cos\theta}{\sin\theta}b_{1}\right)
\end{multline}
\begin{multline}
\dfrac{\partial\mathcal{T}_{2}^{C}\left(M\right)}{\partial x}\left(t,x,y\right)={\textstyle {\displaystyle \int_{0}^{\frac{1}{c\cos\theta}y-\frac{a_{2}}{c\cos\theta}}}}-{\textstyle \dfrac{\partial Q\left(M\right)}{\partial x}}\left(\right.s+t-\frac{1}{c\cos\theta}y+\frac{a_{2}}{c\cos\theta},\\
-sc\sin\theta+x+\frac{\sin\theta}{\cos\theta}y-\frac{\sin\theta}{\cos\theta}a_{2},sc\cos\theta+a_{2}\left.\right)ds+\\
{\textstyle \dfrac{\partial N_{2}^{--}}{\partial x}\left(t-\frac{1}{c\cos\theta}y+\frac{a_{2}}{c\cos\theta},x+\frac{\sin\theta}{\cos\theta}y-\frac{\sin\theta}{\cos\theta}a_{2}\right)}\label{eq:looaoolzo-1-1-1}
\end{multline}
\begin{multline}
\dfrac{\partial\mathcal{T}_{2}^{A}\left(M\right)}{\partial y}\left(t,x,y\right)=\\
{\displaystyle \int_{0}^{t}}-\dfrac{\partial Q\left(M\right)}{\partial y}\left(\right.s,x-c\left(s-t\right)\sin\theta,y+c\left(s-t\right)\cos\theta\left.\right)ds+\\
\dfrac{\partial N_{2}^{0}}{\partial y}\left(x+ct\sin\theta,y-ct\cos\theta\right)
\end{multline}
\begin{multline}
\dfrac{\partial\mathcal{T}_{2}^{B}\left(M\right)}{\partial y}\left(t,x,y\right)=\\
{\textstyle {\displaystyle \int_{0}^{-\frac{1}{c\sin\theta}x+\frac{b_{1}}{c\sin\theta}}}}-{\textstyle \dfrac{\partial Q\left(M\right)}{\partial y}}\left(\right.s+t+{\displaystyle \frac{1}{c\sin\theta}}x-\frac{b_{1}}{c\sin\theta},\\
-sc\sin\theta+b_{1},sc\cos\theta+\frac{\cos\theta}{\sin\theta}x+y-\frac{\cos\theta}{\sin\theta}b_{1}\left.\right)ds+\\
{\textstyle \dfrac{\partial N_{2}^{+}}{\partial y}\left(t+{\displaystyle \frac{1}{c\sin\theta}}x-\frac{b_{1}}{c\sin\theta},\frac{\cos\theta}{\sin\theta}x+y-\frac{\cos\theta}{\sin\theta}b_{1}\right)}
\end{multline}
\begin{multline}
\dfrac{\partial\mathcal{T}_{2}^{C}\left(M\right)}{\partial y}\left(t,x,y\right)=\\
-{\textstyle \frac{1}{c\cos\theta}}Q\left(M\right)\left(t,x,y\right)+{\textstyle {\displaystyle \int_{0}^{\frac{1}{c\cos\theta}y-\frac{a_{2}}{c\cos\theta}}}}\left(\right.{\textstyle \frac{1}{c\cos\theta}}{\textstyle \dfrac{\partial Q\left(M\right)}{\partial t}}\left(\right.s+t-\frac{1}{c\cos\theta}y+\frac{a_{2}}{c\cos\theta},\\
-sc\sin\theta+x+\frac{\sin\theta}{\cos\theta}y-\frac{\sin\theta}{\cos\theta}a_{2},sc\cos\theta+a_{2}\left.\right)\\
-{\textstyle \frac{\sin\theta}{\cos\theta}}{\textstyle \dfrac{\partial Q\left(M\right)}{\partial x}}\left(\right.s+t-\frac{1}{c\cos\theta}y+\frac{a_{2}}{c\cos\theta},-sc\sin\theta+x+\frac{\sin\theta}{\cos\theta}y-\frac{\sin\theta}{\cos\theta}a_{2},sc\cos\theta+a_{2}\left.\right)\left.\right)ds+\\
-\frac{1}{c\cos\theta}{\textstyle \dfrac{\partial N_{2}^{--}}{\partial t}\left(t-\frac{1}{c\cos\theta}y+\frac{a_{2}}{c\cos\theta},x+\frac{\sin\theta}{\cos\theta}y-\frac{\sin\theta}{\cos\theta}a_{2}\right)}\\
+{\textstyle {\displaystyle \frac{\sin\theta}{\cos\theta}}}\dfrac{\partial N_{2}^{--}}{\partial x}\left(t-\frac{1}{c\cos\theta}y+\frac{a_{2}}{c\cos\theta},x+\frac{\sin\theta}{\cos\theta}y-\frac{\sin\theta}{\cos\theta}a_{2}\right)
\end{multline}
{*}{*}{*}{*} $\dfrac{\partial\mathcal{T}_{3}\left(M\right)}{\partial t},\dfrac{\partial\mathcal{T}_{3}\left(M\right)}{\partial x},\dfrac{\partial\mathcal{T}_{3}\left(M\right)}{\partial y}$
are defined in $\mathring{\mathscr{P}}$ except perhaps on the planes
given respectively by the equations $x-ct\sin\theta=a_{1};y+ct\cos\theta=b_{2};x\cos\theta+y\sin\theta=a_{1}\cos\theta+b_{2}\sin\theta$
and
\begin{multline}
\left(\dfrac{\partial}{\partial t},\dfrac{\partial}{\partial x},\dfrac{\partial}{\partial y}\right)\mathcal{T}_{3}\left(M\right)\left(t,x,y\right)=\\
\left(\dfrac{\partial}{\partial t},\dfrac{\partial}{\partial x},\dfrac{\partial}{\partial y}\right)\mathcal{T}_{3}^{A}\left(M\right)\left(t,x,y\right)\cdot\mathbb{I}_{\begin{cases}
x-ct\sin\theta>a_{1}\\
y+ct\cos\theta<b_{2}
\end{cases}}\left(t,x,y\right)+\\
\left(\dfrac{\partial}{\partial t},\dfrac{\partial}{\partial x},\dfrac{\partial}{\partial y}\right)\mathcal{T}_{3}^{B}\left(M\right)\left(t,x,y\right)\cdot\mathbb{I}_{\begin{cases}
x-ct\sin\theta<a_{1}\\
x\cos\theta+y\sin\theta<a_{1}\cos\theta+b_{2}\sin\theta
\end{cases}}\left(t,x,y\right)+\\
\left(\dfrac{\partial}{\partial t},\dfrac{\partial}{\partial x},\dfrac{\partial}{\partial y}\right)\mathcal{T}_{3}^{C}\left(M\right)\left(t,x,y\right)\cdot\mathbb{I}_{\begin{cases}
y+ct\cos\theta>b_{2}\\
x\cos\theta+y\sin\theta>a_{1}\cos\theta+b_{2}\sin\theta
\end{cases}}\left(t,x,y\right)\label{eq:ppmms-2-1-1-1}
\end{multline}
where 
\begin{multline}
\dfrac{\partial\mathcal{T}_{3}^{A}\left(M\right)}{\partial t}\left(t,x,y\right)=-Q\left(M\right)\left(t,x,y\right)+{\displaystyle \int_{0}^{t}}\left[\right.c\sin\theta\dfrac{\partial Q\left(M\right)}{\partial x}\\
\left(\right.s,x+c\left(s-t\right)\sin\theta,y-c\left(s-t\right)\cos\theta\left.\right)-c\cos\theta\dfrac{\partial Q\left(M\right)}{\partial y}\\
\left(\right.s,x+c\left(s-t\right)\sin\theta,y-c\left(s-t\right)\cos\theta\left.\right)\left.\right]ds\\
-c\sin\theta\dfrac{\partial N_{3}^{0}}{\partial x}\left(x-ct\sin\theta,y+ct\cos\theta\right)+c\cos\theta\dfrac{\partial N_{3}^{0}}{\partial y}\left(x-ct\sin\theta,y+ct\cos\theta\right)
\end{multline}
\begin{multline}
\dfrac{\partial\mathcal{T}_{3}^{B}\left(M\right)}{\partial t}\left(t,x,y\right)={\textstyle {\displaystyle \int_{0}^{\frac{1}{c\sin\theta}x-\frac{a_{1}}{c\sin\theta}}}}{\textstyle -\dfrac{\partial Q\left(M\right)}{\partial t}}\left(\right.s+t-\frac{1}{c\sin\theta}x+\frac{a_{1}}{c\sin\theta},\\
sc\sin\theta+a_{1},-sc\cos\theta+\frac{\cos\theta}{\sin\theta}x+y-\frac{\cos\theta}{\sin\theta}a_{1}\left.\right)ds+\\
{\textstyle \dfrac{\partial N_{3}^{-}}{\partial t}\left(t-\frac{1}{c\sin\theta}x+\frac{a_{1}}{c\sin\theta},\frac{\cos\theta}{\sin\theta}x+y-\frac{\cos\theta}{\sin\theta}a_{1}\right)}
\end{multline}
\begin{multline}
\dfrac{\partial\mathcal{T}_{3}^{C}\left(M\right)}{\partial t}\left(t,x,y\right)={\textstyle {\displaystyle \int_{0}^{-\frac{1}{c\cos\theta}y+\frac{b_{2}}{c\cos\theta}}}}-{\textstyle \dfrac{\partial Q\left(M\right)}{\partial t}}\left(\right.s+t+\frac{1}{c\cos\theta}y-\frac{b_{2}}{c\cos\theta},\\
sc\sin\theta+x+\frac{\sin\theta}{\cos\theta}y-\frac{\sin\theta}{\cos\theta}b_{2},-sc\cos\theta+b_{2}\left.\right)ds+\\
{\textstyle \dfrac{\partial N_{3}^{++}}{\partial t}\left(t+\frac{1}{c\cos\theta}y-\frac{b_{2}}{c\cos\theta},x+\frac{\sin\theta}{\cos\theta}y-\frac{\sin\theta}{\cos\theta}b_{2}\right)}\label{eq:looaoolzo-1-2-1}
\end{multline}
\begin{multline}
\dfrac{\partial\mathcal{T}_{3}^{A}\left(M\right)}{\partial x}\left(t,x,y\right)=\\
{\displaystyle \int_{0}^{t}}\left[\right.-\dfrac{\partial Q\left(M\right)}{\partial x}\left(\right.s,x+c\left(s-t\right)\sin\theta,y-c\left(s-t\right)\cos\theta\left.\right)ds+\\
\dfrac{\partial N_{3}^{0}}{\partial x}\left(x-ct\sin\theta,y+ct\cos\theta\right)
\end{multline}
\begin{multline}
\dfrac{\partial\mathcal{T}_{3}^{B}\left(M\right)}{\partial x}\left(t,x,y\right)=-{\textstyle \frac{1}{c\sin\theta}}Q\left(M\right)\left(t,x,y\right)+\\
{\textstyle {\displaystyle \int_{0}^{\frac{1}{c\sin\theta}x-\frac{a_{1}}{c\sin\theta}}}}\left(\right.{\textstyle \frac{1}{c\sin\theta}}{\textstyle \dfrac{\partial Q\left(M\right)}{\partial t}}\left(\right.s+t-\frac{1}{c\sin\theta}x+\frac{a_{1}}{c\sin\theta},\\
sc\sin\theta+a_{1},-sc\cos\theta+\frac{\cos\theta}{\sin\theta}x+y-\frac{\cos\theta}{\sin\theta}a_{1}\left.\right)\\
-{\textstyle \frac{\cos\theta}{\sin\theta}}{\textstyle \dfrac{\partial Q\left(M\right)}{\partial y}}\left(\right.s+t-\frac{1}{c\sin\theta}x+\frac{a_{1}}{c\sin\theta},sc\sin\theta+a_{1},\\
-sc\cos\theta+\frac{\cos\theta}{\sin\theta}x+y-\frac{\cos\theta}{\sin\theta}a_{1}\left.\right)\left.\right)ds+\\
-\frac{1}{c\sin\theta}{\textstyle \dfrac{\partial N_{3}^{-}}{\partial t}\left(t-\frac{1}{c\sin\theta}x+\frac{a_{1}}{c\sin\theta},\frac{\cos\theta}{\sin\theta}x+y-\frac{\cos\theta}{\sin\theta}a_{1}\right)}\\
+{\textstyle {\displaystyle \frac{\cos\theta}{\sin\theta}}}\dfrac{\partial N_{3}^{-}}{\partial t}\left(t-\frac{1}{c\sin\theta}x+\frac{a_{1}}{c\sin\theta},\frac{\cos\theta}{\sin\theta}x+y-\frac{\cos\theta}{\sin\theta}a_{1}\right)
\end{multline}
\begin{multline}
\dfrac{\partial\mathcal{T}_{3}^{C}\left(M\right)}{\partial x}\left(t,x,y\right)={\textstyle {\displaystyle \int_{0}^{-\frac{1}{c\cos\theta}y+\frac{b_{2}}{c\cos\theta}}}}-{\textstyle \dfrac{\partial Q\left(M\right)}{\partial x}}\left(\right.s+t+\frac{1}{c\cos\theta}y-\frac{b_{2}}{c\cos\theta},\\
sc\sin\theta+x+\frac{\sin\theta}{\cos\theta}y-\frac{\sin\theta}{\cos\theta}b_{2},-sc\cos\theta+b_{2}\left.\right)ds+\\
{\textstyle \dfrac{\partial N_{3}^{++}}{\partial x}\left(t+\frac{1}{c\cos\theta}y-\frac{b_{2}}{c\cos\theta},x+\frac{\sin\theta}{\cos\theta}y-\frac{\sin\theta}{\cos\theta}b_{2}\right)}\label{eq:looaoolzo-1-1-1-1}
\end{multline}
\begin{multline}
\dfrac{\partial\mathcal{T}_{3}^{A}\left(M\right)}{\partial y}\left(t,x,y\right)=\\
{\displaystyle \int_{0}^{t}}\left[\right.-\dfrac{\partial Q\left(M\right)}{\partial y}\left(\right.s,x+c\left(s-t\right)\sin\theta,y-c\left(s-t\right)\cos\theta\left.\right)ds+\\
\dfrac{\partial N_{3}^{0}}{\partial y}\left(x-ct\sin\theta,y+ct\cos\theta\right)
\end{multline}
\begin{multline}
\dfrac{\partial\mathcal{T}_{3}^{B}\left(M\right)}{\partial y}\left(t,x,y\right)=\\
{\textstyle {\displaystyle \int_{0}^{\frac{1}{c\sin\theta}x-\frac{a_{1}}{c\sin\theta}}}}-{\textstyle \dfrac{\partial Q\left(M\right)}{\partial y}}\left(\right.s+t-\frac{1}{c\sin\theta}x+\frac{a_{1}}{c\sin\theta},sc\sin\theta+a_{1},\\
-sc\cos\theta+\frac{\cos\theta}{\sin\theta}x+y-\frac{\cos\theta}{\sin\theta}a_{1}\left.\right)ds+\\
+{\textstyle \dfrac{\partial N_{3}^{-}}{\partial y}\left(t-\frac{1}{c\sin\theta}x+\frac{a_{1}}{c\sin\theta},\frac{\cos\theta}{\sin\theta}x+y-\frac{\cos\theta}{\sin\theta}a_{1}\right)}
\end{multline}
\begin{multline}
\dfrac{\partial\mathcal{T}_{3}^{C}\left(M\right)}{\partial y}\left(t,x,y\right)={\textstyle \frac{1}{c\cos\theta}}Q\left(M\right)\left(t,x,y\right)+\\
{\textstyle {\displaystyle \int_{0}^{-\frac{1}{c\cos\theta}y+\frac{b_{2}}{c\cos\theta}}}}\left(\right.-{\textstyle \frac{1}{c\cos\theta}}{\textstyle \dfrac{\partial Q\left(M\right)}{\partial t}}\left(\right.s+t+\frac{1}{c\cos\theta}y-\frac{b_{2}}{c\cos\theta},\\
sc\sin\theta+x+\frac{\sin\theta}{\cos\theta}y-\frac{\sin\theta}{\cos\theta}b_{2},-sc\cos\theta+b_{2}\left.\right)\\
-{\textstyle \frac{\sin\theta}{\cos\theta}}{\textstyle \dfrac{\partial Q\left(M\right)}{\partial x}}\left(\right.s+t+\frac{1}{c\cos\theta}y-\frac{b_{2}}{c\cos\theta},sc\sin\theta+x+\frac{\sin\theta}{\cos\theta}y-\frac{\sin\theta}{\cos\theta}b_{2},\\
-sc\cos\theta+b_{2}\left.\right)\left.\right)ds+\frac{1}{c\cos\theta}{\textstyle \dfrac{\partial N_{3}^{++}}{\partial t}\left(t+\frac{1}{c\cos\theta}y-\frac{b_{2}}{c\cos\theta},x+\frac{\sin\theta}{\cos\theta}y-\frac{\sin\theta}{\cos\theta}b_{2}\right)}\\
+{\textstyle {\displaystyle \frac{\sin\theta}{\cos\theta}}}\dfrac{\partial N_{3}^{++}}{\partial x}\left(t+\frac{1}{c\cos\theta}y-\frac{b_{2}}{c\cos\theta},x+\frac{\sin\theta}{\cos\theta}y-\frac{\sin\theta}{\cos\theta}b_{2}\right)
\end{multline}

{*}{*}{*}{*} $\dfrac{\partial\mathcal{T}_{4}\left(M\right)}{\partial t},\dfrac{\partial\mathcal{T}_{4}\left(M\right)}{\partial x},\dfrac{\partial\mathcal{T}_{4}\left(M\right)}{\partial y}$
are defined in $\mathring{\mathscr{P}}$ except perhaps on the planes
given respectively by the equations $x+ct\cos\theta=b_{1};y+ct\sin\theta=b_{2};x\sin\theta-y\cos\theta=b_{1}\sin\theta-b_{2}\cos\theta$
and
\begin{multline}
\left(\dfrac{\partial}{\partial t},\dfrac{\partial}{\partial x},\dfrac{\partial}{\partial y}\right)\mathcal{T}_{4}\left(M\right)\left(t,x,y\right)=\\
\left(\dfrac{\partial}{\partial t},\dfrac{\partial}{\partial x},\dfrac{\partial}{\partial y}\right)\mathcal{T}_{4}^{A}\left(M\right)\left(t,x,y\right)\cdot\mathbb{I}_{\begin{cases}
x+ct\cos\theta<b_{1}\\
y+ct\sin\theta<b_{2}
\end{cases}}\left(t,x,y\right)+\\
\left(\dfrac{\partial}{\partial t},\dfrac{\partial}{\partial x},\dfrac{\partial}{\partial y}\right)\mathcal{T}_{4}^{B}\left(M\right)\left(t,x,y\right)\cdot\\
\mathbb{I}_{\begin{cases}
x+ct\cos\theta>b_{1}\\
x\sin\theta-y\cos\theta>b_{1}\sin\theta-b_{2}\cos\theta
\end{cases}}\left(t,x,y\right)+\\
\left(\dfrac{\partial}{\partial t},\dfrac{\partial}{\partial x},\dfrac{\partial}{\partial y}\right)\mathcal{T}_{4}^{C}\left(M\right)\left(t,x,y\right)\cdot\mathbb{I}_{\begin{cases}
y+ct\sin\theta>b_{2}\\
x\sin\theta-y\cos\theta<b_{1}\sin\theta-b_{2}\cos\theta
\end{cases}}\left(t,x,y\right)\label{eq:oloooso-1-1}
\end{multline}
where 
\begin{multline}
\dfrac{\partial\mathcal{T}_{4}^{A}\left(M\right)}{\partial t}\left(t,x,y\right)=Q\left(M\right)\left(t,x,y\right)+{\displaystyle \int_{0}^{t}}\left[\right.c\cos\theta\dfrac{\partial Q\left(M\right)}{\partial x}\\
\left(\right.s,x-c\left(s-t\right)\cos\theta,y-c\left(s-t\right)\sin\theta\left.\right)+c\sin\theta\dfrac{\partial Q\left(M\right)}{\partial y}\\
\left(\right.s,x-c\left(s-t\right)\cos\theta,y-c\left(s-t\right)\sin\theta\left.\right)\left.\right]ds\\
+c\cos\theta\dfrac{\partial N_{4}^{0}}{\partial x}\left(x+ct\cos\theta,y+ct\sin\theta\right)+c\sin\theta\dfrac{\partial N_{4}^{0}}{\partial y}\left(x+ct\cos\theta,y+ct\sin\theta\right)
\end{multline}
\begin{multline}
\dfrac{\partial\mathcal{T}_{4}^{B}\left(M\right)}{\partial t}\left(t,x,y\right)={\textstyle {\displaystyle \int_{0}^{-\frac{1}{c\cos\theta}x+\frac{b_{1}}{c\cos\theta}}}}{\textstyle \dfrac{\partial Q\left(M\right)}{\partial t}}\left(\right.s+t+\frac{1}{c\cos\theta}x-\frac{b_{1}}{c\cos\theta},\\
-sc\cos\theta+b_{1},-sc\sin\theta-\frac{\sin\theta}{\cos\theta}x+y+\frac{\sin\theta}{\cos\theta}b_{1}\left.\right)ds+\\
{\textstyle \dfrac{\partial N_{4}^{+}}{\partial t}\left(t+\frac{1}{c\cos\theta}x-\frac{b_{1}}{c\cos\theta},-\frac{\sin\theta}{\cos\theta}x+y+\frac{\sin\theta}{\cos\theta}b_{1}\right)}
\end{multline}
\begin{multline}
\dfrac{\partial\mathcal{T}_{4}^{C}\left(M\right)}{\partial t}\left(t,x,y\right)={\textstyle {\displaystyle \int_{0}^{-\frac{1}{c\sin\theta}y+\frac{b_{2}}{c\sin\theta}}}}{\textstyle \dfrac{\partial Q\left(M\right)}{\partial t}}\left(\right.s+t+\frac{1}{c\sin\theta}y-\frac{b_{2}}{c\sin\theta},\\
-sc\cos\theta+x-\frac{\cos\theta}{\sin\theta}y+\frac{\cos\theta}{\sin\theta}b_{2},-sc\sin\theta+b_{2}\left.\right)ds+\\
{\textstyle \dfrac{\partial N_{4}^{++}}{\partial t}\left(t+\frac{1}{c\sin\theta}y-\frac{b_{2}}{c\sin\theta},x-\frac{\cos\theta}{\sin\theta}y+\frac{\cos\theta}{\sin\theta}b_{2}\right)}\label{eq:looaoolzo-1-3}
\end{multline}
\begin{multline}
\dfrac{\partial\mathcal{T}_{4}^{A}\left(M\right)}{\partial x}\left(t,x,y\right)=\\
{\displaystyle \int_{0}^{t}}\left[\right.\dfrac{\partial Q\left(M\right)}{\partial x}\left(\right.s,x-c\left(s-t\right)\cos\theta,y-c\left(s-t\right)\sin\theta\left.\right)ds+\\
\dfrac{\partial N_{4}^{0}}{\partial x}\left(x+ct\cos\theta,y+ct\sin\theta\right)
\end{multline}
\begin{multline}
\dfrac{\partial\mathcal{T}_{4}^{B}\left(M\right)}{\partial x}\left(t,x,y\right)=-{\textstyle \frac{1}{c\cos\theta}}Q\left(M\right)\left(t,x,y\right)+\\
{\textstyle {\displaystyle \int_{0}^{-\frac{1}{c\cos\theta}x+\frac{b_{1}}{c\cos\theta}}}}\left(\right.{\textstyle \frac{1}{c\cos\theta}}{\textstyle \dfrac{\partial Q\left(M\right)}{\partial t}}\left(\right.s+t+\frac{1}{c\cos\theta}x-\frac{b_{1}}{c\cos\theta},\\
-sc\cos\theta+b_{1},-sc\sin\theta-\frac{\sin\theta}{\cos\theta}x+y+\frac{\sin\theta}{\cos\theta}b_{1}\left.\right)\\
-{\textstyle \frac{\sin\theta}{\cos\theta}}{\textstyle \dfrac{\partial Q\left(M\right)}{\partial y}}\left(\right.s+t+\frac{1}{c\cos\theta}x-\frac{b_{1}}{c\cos\theta},-sc\cos\theta+b_{1},\\
-sc\sin\theta-\frac{\sin\theta}{\cos\theta}x+y+\frac{\sin\theta}{\cos\theta}b_{1}\left.\right)\left.\right)ds+\\
\frac{1}{c\cos\theta}{\textstyle \dfrac{\partial N_{4}^{+}}{\partial t}\left(t+\frac{1}{c\cos\theta}x-\frac{b_{1}}{c\cos\theta},-\frac{\sin\theta}{\cos\theta}x+y+\frac{\sin\theta}{\cos\theta}b_{1}\right)}\\
-{\textstyle {\displaystyle \frac{\sin\theta}{\cos\theta}}}\dfrac{\partial N_{4}^{+}}{\partial y}\left(t+\frac{1}{c\cos\theta}x-\frac{b_{1}}{c\cos\theta},-\frac{\sin\theta}{\cos\theta}x+y+\frac{\sin\theta}{\cos\theta}b_{1}\right)
\end{multline}
\begin{multline}
\dfrac{\partial\mathcal{T}_{4}^{C}\left(M\right)}{\partial x}\left(t,x,y\right)={\textstyle {\displaystyle \int_{0}^{-\frac{1}{c\sin\theta}y+\frac{b_{2}}{c\sin\theta}}}}{\textstyle \dfrac{\partial Q\left(M\right)}{\partial x}}\left(\right.s+t+\frac{1}{c\sin\theta}y-\frac{b_{2}}{c\sin\theta},\\
-sc\cos\theta+x-\frac{\cos\theta}{\sin\theta}y+\frac{\cos\theta}{\sin\theta}b_{2},-sc\sin\theta+b_{2}\left.\right)ds+\\
{\textstyle \dfrac{\partial N_{4}^{++}}{\partial x}\left(t+\frac{1}{c\sin\theta}y-\frac{b_{2}}{c\sin\theta},x-\frac{\cos\theta}{\sin\theta}y+\frac{\cos\theta}{\sin\theta}b_{2}\right)}\label{eq:looaoolzo-1-1-2}
\end{multline}
\begin{multline}
\dfrac{\partial\mathcal{T}_{4}^{A}\left(M\right)}{\partial y}\left(t,x,y\right)=\\
{\displaystyle \int_{0}^{t}}\left[\right.\dfrac{\partial Q\left(M\right)}{\partial y}\left(\right.s,x-c\left(s-t\right)\cos\theta,y-c\left(s-t\right)\sin\theta\left.\right)ds+\\
\dfrac{\partial N_{4}^{0}}{\partial y}\left(x+ct\cos\theta,y+ct\sin\theta\right)
\end{multline}
\begin{multline}
\dfrac{\partial\mathcal{T}_{4}^{B}\left(M\right)}{\partial y}\left(t,x,y\right)=\\
{\textstyle {\displaystyle \int_{0}^{-\frac{1}{c\cos\theta}x+\frac{b_{1}}{c\cos\theta}}}}{\textstyle \dfrac{\partial Q\left(M\right)}{\partial y}}\left(\right.s+t+\frac{1}{c\cos\theta}x-\frac{b_{1}}{c\cos\theta},\\
-sc\cos\theta+b_{1},-sc\sin\theta-\frac{\sin\theta}{\cos\theta}x+y+\frac{\sin\theta}{\cos\theta}b_{1}\left.\right)ds+\\
{\textstyle \dfrac{\partial N_{4}^{+}}{\partial y}\left(t+\frac{1}{c\cos\theta}x-\frac{b_{1}}{c\cos\theta},-\frac{\sin\theta}{\cos\theta}x+y+\frac{\sin\theta}{\cos\theta}b_{1}\right)}
\end{multline}
\begin{multline}
\dfrac{\partial\mathcal{T}_{4}^{C}\left(M\right)}{\partial y}\left(t,x,y\right)=\\
-{\textstyle \frac{1}{c\sin\theta}}Q\left(M\right)\left(t,x,y\right)+{\textstyle {\displaystyle \int_{0}^{-\frac{1}{c\sin\theta}y+\frac{b_{2}}{c\sin\theta}}}}\left(\right.{\textstyle \frac{1}{c\sin\theta}}{\textstyle \dfrac{\partial Q\left(M\right)}{\partial t}}\left(\right.s+t+\frac{1}{c\sin\theta}y-\frac{b_{2}}{c\sin\theta},\\
-sc\cos\theta+x-\frac{\cos\theta}{\sin\theta}y+\frac{\cos\theta}{\sin\theta}b_{2},-sc\sin\theta+b_{2}\left.\right)\\
-{\textstyle \frac{\cos\theta}{\sin\theta}}{\textstyle \dfrac{\partial Q\left(M\right)}{\partial x}}\left(\right.s+t+\frac{1}{c\sin\theta}y-\frac{b_{2}}{c\sin\theta},-sc\cos\theta+x-\frac{\cos\theta}{\sin\theta}y+\frac{\cos\theta}{\sin\theta}b_{2},\\
-sc\sin\theta+b_{2}\left.\right)\left.\right)ds+\\
+\frac{1}{c\sin\theta}{\textstyle \dfrac{\partial N_{4}^{++}}{\partial t}\left(t+\frac{1}{c\sin\theta}y-\frac{b_{2}}{c\sin\theta},x-\frac{\cos\theta}{\sin\theta}y+\frac{\cos\theta}{\sin\theta}b_{2}\right)}\\
-{\textstyle {\displaystyle \frac{\cos\theta}{\sin\theta}}}\dfrac{\partial N_{4}^{++}}{\partial x}\left(t+\frac{1}{c\sin\theta}y-\frac{b_{2}}{c\sin\theta},x-\frac{\cos\theta}{\sin\theta}y+\frac{\cos\theta}{\sin\theta}b_{2}\right)\label{eq:ikikie}
\end{multline}

\end{document}